\documentclass[11 pt, reqno]{amsart}
\usepackage{amssymb,amsmath}
\usepackage{hyperref, cite}
\usepackage{epsfig}
\newtheorem{theorem}{Theorem}[section]
\newtheorem{corollary}[theorem]{Corollary}

\newtheorem{proposition}[theorem]{Proposition}
\theoremstyle{definition}

\newtheorem{question}[theorem]{Question}
\newtheorem{problem}[theorem]{Problem}
\newtheorem{definition}[theorem]{Definition}
\newtheorem{example}[theorem]{Example}
\newtheorem{remark}[theorem]{Remark}
\newtheorem{notation}[theorem]{Notation}
\numberwithin{equation}{subsection}
\newtheorem*{ack}{Acknowledgement}
\usepackage[all,cmtip]{xy}
\usepackage{float}
\usepackage{multicol}

\newcommand{\Aut}{\operatorname{Aut}}

\newcommand{\RNum}[1]{\uppercase\expandafter{\romannumeral #1\relax}}
\usepackage[autostyle]{csquotes}
\usepackage{color}

\begin{document}
\title{Virtually symmetric representations and marked Gauss diagrams}

\author{Valeriy G. Bardakov}
\address{Sobolev Institute of Mathematics, 4 Acad. Koptyug avenue, 630090, Novosibirsk, Russia.}
\address{Novosibirsk State  University, 2 Pirogova Street, 630090, Novosibirsk, Russia.}
\address{Novosibirsk State Agrarian University, Dobrolyubova street, 160, Novosibirsk, 630039, Russia.}
\address{Regional Scientific and Educational Mathematical Center of Tomsk State University, 36 Lenin Ave., Tomsk, Russia.}
\email{bardakov@math.nsc.ru}

\author{Mikhail V. Neshchadim}
\address{Sobolev Institute of Mathematics, 4 Acad. Koptyug avenue, 630090, Novosibirsk, Russia.}
\address{Novosibirsk State  University, 2 Pirogova Street, 630090, Novosibirsk, Russia.}
\address{Regional Scientific and Educational Mathematical Center of Tomsk State University, 36 Lenin Ave., Tomsk, Russia}
\email{neshch@math.nsc.ru}

\author{Manpreet Singh}
\address{Department of Mathematical Sciences, Indian Institute of Science Education and Research (IISER) Mohali, Sector 81,  S. A. S. Nagar, P. O. Manauli, Punjab 140306, India.}
\email{manpreet.math23@gmail.com}


\begin{abstract}
In this paper, we define the notion of a virtually symmetric representation of representations of virtual braid groups and prove that many known representations are equivalent to virtually symmetric. Using one such representation, we define the notion of virtual link groups which is an extension of virtual link groups defined by Kauffman. Moreover, we introduce the concept of marked Gauss diagrams as a generalisation of Gauss diagrams and their interpretation in terms of knot-like diagrams. We extend the definition of virtual link groups to marked Gauss diagrams and define their peripheral structure. We define $C_m$-groups and prove that every group presented by a $1$-irreducible $C_1$-presentation of deficiency $1$ or $2$ can be realized as the group of a marked Gauss diagram.
\end{abstract}
\maketitle

\section{Introduction}

A generic projection of a link onto a plane with the information of over and under crossing arcs at double points is called its link diagram. Two link diagrams represent the same link if and only if they are related by a finite sequence of Reidemeister moves and planar isotopies. Kauffman \cite{Kauffman-1} introduced virtual links as a generalisation of classical links. A virtual link diagram is a generic immersion of a finite number of oriented circles into a plane having finitely many real and virtual crossings, where a virtual crossing is depicted by a small circle on the transversal intersection of two arcs. Two virtual link diagrams are equivalent if one can be obtained from the other by a finite sequence of generalised Reidemeister moves and isotopies of the plane. Virtual links can also be realized as link diagrams on oriented surfaces known as abstract link diagrams \cite{KK-1}, and that every virtual link can be uniquely represented by a link in a thickened, compact and oriented surface with a condition on the link complement \cite{Kuperberg-1}. The role of classical braid groups in virtual knot theory is played by virtual braid groups \cite{Kauffman-1}. S. Kamada \cite{Kamada-1} proved an analogue of Alexander and Markov theorems for oriented virtual links. An independent approach to Alexander and Markov theorems was given by Kauffman and Lambropoulou \cite{KL-1}. 
\medskip

Gauss diagrams are another way of studying links from a combinatorial point of view. It is well-known that for every oriented link diagram there is a corresponding Gauss diagram. However, there are Gauss diagrams which do not represent any link diagram. Similar to the classical setting, to every virtual link diagram, one can associate a Gauss diagram by ignoring the virtual crossings. Goussarov, Polyak and Viro \cite{GPV-1} considered Gauss diagrams up to a finite sequence of abstract Reidemeister moves for Gauss diagrams. Contrary to the classical setting, there is a one-to-one correspondence between the set of equivalence classes of virtual link diagrams and the set of equivalence classes of Gauss diagrams. By considering Gauss diagrams, finite-type invariants of virtual links can be described (see \cite{GPV-1} for more details).
\medskip

One of the fundamental problems in knot theory is the classification of knots. It has been established that the knot group of a classical knot and its peripheral subgroup along with the meridian is a complete knot invariant up to the orientation of the knot and the ambient space. Kauffman \cite{Kauffman-1} extended the notion of knot group and knot quandle to the virtual setting via diagrams, whose topological interpretation is given in \cite{FRS-1, KK-1}. Since then, various definitions of virtual link groups have been introduced in the literature \cite{Bar-1,BB-1, BB-2, BMN-1, BMN-2, BN-2, BN-3, BDGGHN-1, CSW-1, Kauffman-1, Man, SW-1}, some of which use representations of virtual braid groups into the automorphisms of appropriate groups or modules.
\medskip 

In this paper, we have introduced the notion of a virtually symmetric representation of representations of virtual braid groups and prove that most of the known representations are virtually symmetric. The key advantage of a representation being virtually symmetric is that once the virtual link group is defined, it can be described using Gauss diagrams. We have also constructed a linear, local, non-homogeneous representation of virtual braid groups and proved that it is also equivalent to a virtually symmetric representation. 
\medskip

The definition of virtual link groups introduced in this paper is constructed from a representation of virtual braid groups defined in \cite{BMN-1} which is equivalent to a virtually symmetric representation. Thus we describe these groups using Gauss diagrams too. 
\medskip

On observing the presentations of virtual link groups constructed using Gauss diagrams we introduce the notion of $C_m$-groups, which are a particular type of $C$-groups defined by Kulikov \cite{Kulikov-1}. Furthermore, we introduce marked Gauss diagrams as a generalisation of Gauss diagrams upon which we define an equivalence relation generated by a sequence of moves called marked Reidemeister moves. The set of equivalence classes of Gauss diagrams canonically injects into the set of equivalence classes of marked Gauss diagrams.
\medskip

Beineke and Harary \cite{BH-1} introduced marked digraphs as directed graphs with an information of positive or negative signs on nodes $($vertices$)$. In this paper, we consider marked cycles and define a marked virtual link diagram as a generic immersion of these cycles into a plane with virtual and classical crossing information at double points. There is a one-to-one correspondence between marked Gauss diagrams and marked virtual link diagrams under the equivalence relation generated by moves in the corresponding sets.
\medskip

Virtual link diagrams having node points (without signs) are also introduced in \cite{Dye-1}. There node points are named as (unoriented) cut points and are used to extend the notion of checkerboard colorings to all virtual link diagrams. Lately, N. Kamada \cite{Naoko-1} generalised the notion of cut points to oriented cut points and used it to construct a map from the set of virtual links to the set of (mod $m$) almost classical virtual links. Oriented cut points were further divided into coherent and incoherent cut points, and thus one can think of oriented cut points as signed node points. For more on cut points, we refer the reader to \cite{Dye-1, Dye-2, Naoko-1, Naoko-2}.
\medskip

To each marked Gauss diagram, we assign a group, which is isomorphic to the virtual link group if the considered marked Gauss diagram is a Gauss diagram. Moreover, we define its peripheral subgroups and consequently the peripheral structure of a marked Gauss diagram. We study these groups based on the methods and results developed by Kim \cite{Kim-1}. Kim \cite{Kim-1} had studied virtual link groups and their peripheral structures, which were defined by Kauffman \cite{Kauffman-1}. The main result we have proved is that every group with a $1$-irreducible $C_1$-presentation of deficiency $1$ or $2$ can be realized as a group of some marked Gauss diagram.
\medskip

The paper is organised as follows. In Section \ref{S: prelim}, we recall some definitions and fix some notations. In Section \ref{S: Virtually symmetric representations}, we define the notion of a virtually symmetric representation for representations of virtual braid groups and prove that many known representations are equivalent to virtually symmetric, including the one we have constructed in the section. In Section \ref{S: Virtual link groups}, we define virtual link groups using a particular virtually symmetric representation. In Section \ref{S: marked Gauss Diagrams}, we introduce marked Gauss diagrams and extend the notion of virtual link groups to them. In Section \ref{S: Marked virtual link diagrams}, we give an interpretation of marked Gauss diagrams in terms of generic immersion of marked cycles into a plane and call such diagrams as marked virtual link diagrams. In Section \ref{S: Realization-of-irreducible $C_1$-group}, we prove that every $1$-irreducible $C_1$-group presenting a $1$-irreducible $C_1$-presentation of deficiency $1$ or $2$ can be realized as a group of some marked Gauss diagram. In Section \ref{S: Peripheral subgroup and peripheral structure}, we define the notion of meridian, longitude, peripheral subgroup and peripheral structure for marked Gauss diagrams and study their algebraic properties for $1$-circle marked Gauss diagrams. In Section \ref{S: Peripherally specified homomorphs}, we study the peripherally specified homomorphic image of groups associated to marked Gauss diagrams. In Section \ref{S: Problems}, we summarize the paper and conclude with some questions for future research. The results in sections \ref{S: Realization-of-irreducible $C_1$-group}, \ref{S: Peripheral subgroup and peripheral structure} and \ref{S: Peripherally specified homomorphs} are inspired and modelled on the work of Kim \cite{Kim-1}.
\medskip

\section{Preliminaries}\label{S: prelim}

The \textit{$n$-strand virtual braid group} $VB _n$ is the group
with a presentation having generators $\sigma_1, \sigma_2, \ldots, \sigma_{n-1},\rho_1, \ldots, \rho_{n-1}$ and following set of relations:

\begin{itemize}
\item relations of the braid group on $n$ strands:
\begin{align*}
\sigma_i \sigma_{i+1} \sigma_i &= \sigma_{i+1} \sigma_i \sigma_{i+1} ~~\textrm{for}~~i \in \{1,2,\ldots,n-2\},\\
\sigma_i \sigma_j &= \sigma_j \sigma_i ~~\textrm{for}~~ |i-j| \geq 2~~ \textrm{and}~~i,j \in \{1,2,\ldots, n-1\},
\end{align*}

\item relations of the symmetric group:
\begin{align*}
\rho_i ^2 &=1 ~~\textrm{for}~~i \in \{1,2,\ldots,n-1\},\\
\rho_i \rho_j&= \rho_j \rho_i ~~\textrm{for}~~ |i-j| \geq 2 ~~\textrm{and}~~i,j \in \{1,2,\ldots, n-1\},\\
\rho_i \rho_{i+1} \rho_{i} &= \rho_{i+1} \rho_{i} \rho_{i+1} ~~\textrm{for}~~ i \in \{1,2,\ldots, n-2\},
\end{align*}

\item mixed relations:
\begin{align*}
\sigma_i \rho_j &= \rho_j \sigma_i~~\textrm{for}~~ |i-j| \geq 2~~\textrm{and}~~i,j \in \{1,2,\ldots,n-1\},\\
\rho_i \rho_{i+1} \sigma_i &= \sigma_{i+1} \rho_i \rho_{i+1}~~\textrm{for}~~ i \in \{1,2,\ldots,n-2\}.
\end{align*}

\end{itemize}

\begin{notation}
Let $G$ be a group and $a,b \in G$. Then $a^b$ denotes the element $b^{-1}a b$ in $G$.
\end{notation}

\begin{notation}
Let $G_1$ and $G_2$ be two groups. Then $G_1 * G_2$ denotes the free product of $G_1$ and $G_2$.
\end{notation}

\begin{notation}
Throughout the paper $\epsilon$ and $\eta$ are either $+1$ or $-1$.
\end{notation}

\section{Virtually symmetric representations}\label{S: Virtually symmetric representations}

In this section, we define the notion of a virtually symmetric representation for representations of virtual braid groups, and prove that some known representations are equivalent to virtually symmetric representations.

\begin{definition}
A representation $\varphi : VB_n \to \Aut(H)$ of the virtual braid group $VB_n$ into the automorphism group of some group (or module) $H=\langle h_1, h_2, \ldots, h_m~~|~~\mathcal{R}\rangle$ is called {\it virtually symmetric} if for any generator $\rho_i$, $i=1, 2, \ldots, n-1$, its image $\varphi(\rho_i)$ is a permutation of the generators  $h_1, h_2, \ldots, h_m$.
\end{definition}

Let $\varphi_i: VB_n \to \Aut(H)$ $(i=1,2)$ be two representations. We say that $\varphi_1$ and $\varphi_2$ are {\it equivalent} if there exists an automorphism $\phi: H \to H$ such that $\varphi_1(\beta)= \phi^{-1} \circ \varphi_2(\beta) \circ \phi$ for all $\beta \in VB_n$.
\medskip

Let $F_{n,n} = F_n \ast \mathbb{Z}^n$, where $F_n=\langle x_1, x_2, \ldots,x_n \rangle$ is the free group of rank $n$ and
$\mathbb{Z}^n = \langle v_1, v_2, \ldots, v_n \rangle$ is the free abelian group of rank $n$.
In \cite[Theorem 4.1]{BMN-1}, the extension $\varphi_M$ of Artin representation is defined for virtual braid groups, where $\varphi_M: VB_n \rightarrow \Aut(F_{n,n})$ is defined by its action on the generators as follows

\begin{align*}
	&\varphi_M(\sigma_i) :
	\left\{
	\begin{array}{l}
		x_i \mapsto  x_i x_{i+1} x_i^{-1}, \\
		x_{i+1} \mapsto x_i,  \\
		x_j \mapsto x_j, \textrm{ for } j \neq i, i+1, \\
		v_i \mapsto  v_{i+1}, \\
		v_{i+1} \mapsto v_i,  \\
		v_j \mapsto v_j , \textrm{ for } j \neq i, i+1,
	\end{array}
	\right.~~~
	&\varphi_M(\rho_i) :
	\left\{
	\begin{array}{l}
		x_i \mapsto  x_{i+1}^{v_i^{-1}}, \\
		x_{i+1} \mapsto ~~x_i^{v_{i+1}},  \\
		x_j \mapsto x_j, \textrm{ for } j \neq i, i+1, \\
		v_i \mapsto  v_{i+1}, \\
		v_{i+1} \mapsto v_i,\\
		v_j \mapsto v_j , \textrm{ for } j \neq i, i+1.
	\end{array}
	\right.
\end{align*}

In particular, if we put $v_1 = \cdots = v_n = 1$, we get the representation $\varphi_0 : VB_n \to \Aut(F_{n})$ defined by Vershinin \cite{VV-1}. Hereafter, we only write non-trivial actions on generators assuming that all other
generators are fixed.
\medskip

\begin{proposition}\label{P: main-representation-in-file}
The representation $\varphi_M: VB_n \rightarrow \Aut(F_{n,n})$
 is equivalent to a virtually symmetric representation.
\end{proposition}

\begin{proof}
We define an automorphism $\phi: F_{n,n} \to F_{n,n}$ by setting
\begin{align*}
\phi(x_i)&=x_i ^{(v_i v_{i+1} \ldots v_n)}, \,\,\, i=1,\ldots,n,\\
\phi(v_i)&=v_i,\,\,\, i=1,\ldots,n.
\end{align*}
Thus, we have a new representation $\varphi_S: VB_n \to \Aut(F_{n,n})$ of the virtual braid group $VB_n$ into the automorphism group of $F_{n,n}$ by setting
$$
\varphi_S(\beta)=\phi^{-1} \circ \varphi_M(\beta) \circ \phi, ~~\textrm{for}~~ \beta \in VB_n.
$$
In particular,

\begin{align*}
	\varphi_S(\sigma_i) :
	\left\{
	\begin{array}{l}
		x_i \mapsto  x_i ~~x_{i+1}^{v_i}~~ x_i^{-1}, \\
		x_{i+1} \mapsto x_i^{v_{i+1}^{-1}},  \\
		v_i \mapsto  v_{i+1}, \\
		v_{i+1} \mapsto v_i,
	\end{array}
	\right.~~~
	\varphi_S(\rho_i) :
	\left\{
	\begin{array}{l}
		x_i \mapsto  x_{i+1}, \\
		x_{i+1} \mapsto x_i,  \\
		v_i \mapsto  v_{i+1}, \\
		v_{i+1} \mapsto v_i .
	\end{array}
	\right.
\end{align*}

Notice that

$$
\varphi_S(\sigma_i^{-1}) :
\left\{
\begin{array}{l}
	x_i \mapsto  x_{i+1}^{v_i}, \\
	x_{i+1} \mapsto x_{i+1} ^{-v_i v_{i+1} ^{-1}}~~ x_i^{v_{i+1} ^{-1}}~~ x_{i+1} ^{v_i v_{i+1}^{-1}},  \\
	v_i \mapsto  v_{i+1}, \\
	v_{i+1} \mapsto v_i.
\end{array}
\right.
$$

It follows that $\varphi_S$ is a virtually symmetric representation.
\end{proof}
Here we would like to mention that we shall use the representations $\varphi_M$ and $\varphi_S$ in Section \ref{S: Virtual link groups} to define virtual link groups.

\subsection{Generalised Artin representation}\label{SS: generalised Artin representation}

Let $F_{n+1} =\langle x_1, x_2, \ldots,x_n, v \rangle$ be the free group of rank $n+1$.
In \cite{Bar-1,Man}, a representation  $\varphi_A: VB_n \rightarrow \Aut(F_{n+1})$ is defined by its action on the generators as follows
\begin{align*}
&\varphi_A(\sigma_i) :
\left\{
\begin{array}{l}
  x_i \mapsto  x_i x_{i+1} x_i^{-1}, \\
  x_{i+1} \mapsto x_i,  \\
\end{array}
\right.~~~
&\varphi_A(\rho_i) :
\left\{
\begin{array}{l}
  x_i \mapsto  x_{i+1}^{v^{-1}}, \\
  x_{i+1} \mapsto x_i^{v}.  \\
\end{array}
\right.
\end{align*}
\medskip

We define an automorphism $\phi: F_{n+1} \to F_{n+1}$ by setting
\begin{align*}
\phi(x_i)&=x_i ^{v^{n-i}},\,\,\, i=1,\ldots,n,\\
\phi(v)&=v.
\end{align*}
Thus, we have a new representation $\tilde{\varphi}_A: VB_n \to \Aut(F_{n+1})$ by setting
$$
\tilde{\varphi}_A(\beta)=\phi^{-1} \circ \varphi_A(\beta) \circ \phi, ~~\textrm{where}~~ \beta \in VB_n.
$$
In particular,
\begin{align*}
\tilde{\varphi}_A(\sigma_i) :&
\left\{
\begin{array}{l}
  x_i \mapsto  x_i ~~x_{i+1}^{v}~~ x_i^{-1}, \\
  x_{i+1} \mapsto x_i^{v^{-1}},  \\
\end{array}
\right.~~~
&\tilde{\varphi}_A(\rho_i) :&
\left\{
\begin{array}{l}
  x_i \mapsto  x_{i+1}, \\
  x_{i+1} \mapsto x_i.  \\
\end{array}
\right.
\end{align*}
Therefore, $\varphi_A$ is equivalent to a virtually symmetric representation.

\medskip
\subsection{Silver-Williams representation}\label{SS: The Silver-Williams representation}
Let $F_{n,n+1} = F_n \ast \mathbb{Z}^{n+1}$, where

$F_n=\langle x_1, x_2, \ldots,x_n \rangle$ is the free group of rank $n$ and
$\mathbb{Z}^{n+1} = \langle u_1, u_2, \ldots, u_n, v \rangle$ is the free abelian group of rank $n+1$.
Using the definition of the generalised Alexander group for virtual links \cite{SW-1},
a representation $\varphi_{SW} : VB_n \longrightarrow \mathrm{Aut}(F_{n,n+1})$ is constructed in \cite{BMN-1} which is defined by its action on the generators as follows

 \begin{align*}
	&\varphi_{SW}(\sigma_i) :
	\left\{
	\begin{array}{l}
		x_i \mapsto  x_i x_{i+1}^{u_i} x_i^{-vu_{i+1}}, \\
		x_{i+1} \mapsto x_i^v,  \\
		u_i \mapsto  u_{i+1}, \\
		u_{i+1} \mapsto u_i,
	\end{array}
	\right.~~~
	&\varphi_{SW}(\rho_i) :
	\left\{
	\begin{array}{l}
		x_i \mapsto  x_{i+1}, \\
		x_{i+1} \mapsto x_i,  \\
		u_i \mapsto  u_{i+1}, \\
		u_{i+1} \mapsto u_i.
	\end{array}
	\right.
\end{align*}

This representation is virtually symmetric. 


\subsection{Boden-Dies representation}\label{SS: Boden-Dies representation}
Let $F_{n,2} = F_n \ast \mathbb{Z}^2$, where $F_n=\langle x_1, x_2, \ldots,x_n \rangle$ is the free group of rank $n$ and
$\mathbb{Z}^2 = \langle u, v \rangle$ is the free abelian group of rank $2$.
In \cite{BDGGHN-1}, a representation  $\varphi_{BD}: VB_n \rightarrow \Aut(F_{n,2})$ is defined by its action on the generators as follows
\begin{align*}
&\varphi_{BD}(\sigma_i) :
\left\{
\begin{array}{l}
  x_i \mapsto  x_i x_{i+1} x_i^{-u}, \\
  x_{i+1} \mapsto x_i^{u},  \\
\end{array}
\right.~~~
&\varphi_{BD}(\rho_i) :
\left\{
\begin{array}{l}
  x_i \mapsto  x_{i+1}^{v^{-1}}, \\
  x_{i+1} \mapsto x_i^{v}.  \\
\end{array}
\right.
\end{align*}

We define an automorphism $\phi: F_{n,2} \to F_{n,2}$ by setting
\begin{align*}
\phi(x_i)&=x_i ^{v^{n-i}},\textrm{ for } i=1,\ldots,n,\\
\phi(u)&=u,\\
\phi(v)&=v.
\end{align*}
Thus, we have a new representation $\tilde{\varphi}_{BD}: VB_n \to \Aut(F_{n,2})$ by defining
$$
\tilde{\varphi}_{BD}(\beta)=\phi^{-1} \circ \varphi_{BD}(\beta) \circ \phi, ~~\textrm{where}~~ \beta \in VB_n.
$$
We see that
\begin{align*}
\tilde{\varphi}_{BD}(\sigma_i) :&
\left\{
\begin{array}{l}
  x_i \mapsto  x_i ~~x_{i+1}^{v}~~ x_i^{-u}, \\
  x_{i+1} \mapsto x_i^{uv^{-1}},  \\
\end{array}
\right.~~~
&\tilde{\varphi}_{BD}(\rho_i) :&
\left\{
\begin{array}{l}
  x_i \mapsto  x_{i+1}, \\
  x_{i+1} \mapsto x_i.  \\
\end{array}
\right.
\end{align*}
Therefore, $\varphi_{BD}$ is equivalent to a virtually symmetric representation.


\subsection{Extended Wada representations}\label{SS: generalised Wada representation}
Let $F_{n,n} = F_n \ast \mathbb{Z}^n$, where $F_n=\langle x_1, x_2, \ldots,x_n \rangle$ is the free group of rank $n$ and
$\mathbb{Z}^n = \langle v_1, v_2, \ldots, v_n \rangle$ is the free abelian group of rank $n$.
Wada \cite{Wada-1} defined representations $w_{1,r}, w_2$,  $ w_3$, where  $r \in \mathbb{Z}$, of the braid group $B_n$ into $\Aut(F_{n})$. We define extensions of Wada representations
$w: VB_n \rightarrow \Aut(F_{n,n})$, where $w=w_{1,r}, w_2$ or $ w_3$ to the virtual braid group $VB _n$ by its action on the generators as follows

\begin{align*}
	&w_{1,r}(\sigma_i) :
	\left\{
	\begin{array}{l}
		x_i \mapsto  x_i^r x_{i+1} x_i^{-r}, \\
		x_{i+1} \mapsto x_i,  \\
		v_i \mapsto  v_{i+1}, \\
		v_{i+1} \mapsto v_i,
	\end{array}
	\right.~~~
	&w_{1,r}(\rho_i) :
	\left\{
	\begin{array}{l}
		x_i \mapsto  x_{i+1}^{v_i^{-1}}, \\
		x_{i+1} \mapsto x_i^{v_{i+1}},  \\
		v_i \mapsto  v_{i+1}, \\
		v_{i+1} \mapsto v_i ,
	\end{array}
	\right.
\end{align*}

\begin{align*}
	&w_2(\sigma_i) :
	\left\{
	\begin{array}{l}
		x_i \mapsto  x_i x_{i+1}^{-1} x_i, \\
		x_{i+1} \mapsto x_i,  \\
		v_i \mapsto  v_{i+1}, \\
		v_{i+1} \mapsto v_i,
	\end{array}
	\right.~~~
	&w_2(\rho_i) :
	\left\{
	\begin{array}{l}
		x_i \mapsto  x_{i+1}^{v_i^{-1}}, \\
		x_{i+1} \mapsto x_i^{v_{i+1}},  \\
		v_i \mapsto  v_{i+1}, \\
		v_{i+1} \mapsto v_i ,
	\end{array}
	\right.
\end{align*}

\begin{align*}
	&w_3(\sigma_i) :
	\left\{
	\begin{array}{l}
		x_i \mapsto  x_i^2 x_{i+1}, \\
		x_{i+1} \mapsto x_{i+1}^{-1}x_i^{-1}x_{i+1},  \\
		v_i \mapsto  v_{i+1}, \\
		v_{i+1} \mapsto v_i,
	\end{array}
	\right.~~~
	&w_3(\rho_i) :
	\left\{
	\begin{array}{l}
		x_i \mapsto  x_{i+1}^{v_i^{-1}}, \\
		x_{i+1} \mapsto x_i^{v_{i+1}},  \\
		v_i \mapsto  v_{i+1}, \\
		v_{i+1} \mapsto v_i.
	\end{array}
\right.
\end{align*}

The case $v_1 = v_2 = \cdots = v_n$ is studied in \cite{Mih}. We define an automorphism $\phi$ of $F_{n,n}$ by setting
\begin{align*}
\phi(x_i)&=x_i ^{(v_i v_{i+1} \ldots v_n)}\\
\phi(v_i)&=v_i,\,\,\, i=1,\ldots,n.
\end{align*}
Thus, we have new representations $\tilde{w}: VB_n \to \Aut(F_{n+1})$ by defining
$$
\tilde{w}(\beta)=\phi^{-1} \circ w(\beta) \circ \phi, ~~\textrm{where}~~ \beta \in VB_n, ~~~ w=w_{1,r}, w_2~~or~~ w_3.
$$
One can check that

\begin{align*}
	&\tilde{w}_{1,r}(\sigma_i) :
	\left\{
	\begin{array}{l}
		x_i \mapsto  x_i^r x_{i+1}^{v_i} x_i^{-r}, \\
		x_{i+1} \mapsto x_i^{v_{i+1}^{-1}},  \\
		v_i \mapsto  v_{i+1}, \\
		v_{i+1} \mapsto v_i,
	\end{array}
	\right.~~~
	&\tilde{w}_{1,r}(\rho_i) :
	\left\{
	\begin{array}{l}
		x_i \mapsto  x_{i+1}, \\
		x_{i+1} \mapsto x_i,  \\
		v_i \mapsto  v_{i+1}, \\
		v_{i+1} \mapsto v_i , \\
	\end{array}
	\right.
\end{align*}

\begin{align*}
	&\tilde{w}_2(\sigma_i) :
	\left\{
	\begin{array}{l}
		x_i \mapsto  x_i x_{i+1}^{-v_i} x_i, \\
		x_{i+1} \mapsto x_i^{v_{i+1}^{-1}},  \\
		v_i \mapsto  v_{i+1}, \\
		v_{i+1} \mapsto v_i,
	\end{array}
	\right.~~~
	&\tilde{w}_2(\rho_i) :
	\left\{
	\begin{array}{l}
		x_i \mapsto  x_{i+1}, \\
		x_{i+1} \mapsto x_i,  \\
		v_i \mapsto  v_{i+1}, \\
		v_{i+1} \mapsto v_i , \\
	\end{array}
	\right.
\end{align*}

\begin{align*}
	&\tilde{w}_3(\sigma_i) :
	\left\{
	\begin{array}{l}
		x_i \mapsto  x_i^2 x_{i+1}^{v_i}, \\
		x_{i+1} \mapsto x_{i+1}^{-v_iv_{i+1}^{-1}}x_i^{-v_{i+1}^{-1}}x_{i+1}^{v_iv_{i+1}^{-1}},  \\
		v_i \mapsto  v_{i+1}, \\
		v_{i+1} \mapsto v_i,
	\end{array}
	\right.~~~
	&\tilde{w}_3(\rho_i) :
	\left\{
	\begin{array}{l}
		x_i \mapsto  x_{i+1}, \\
		x_{i+1} \mapsto x_i,  \\
			v_i \mapsto  v_{i+1}, \\
		v_{i+1} \mapsto v_i.
	\end{array}
	\right.
\end{align*}

Therefore, the extended Wada representations are equivalent to virtually symmetric representations.

\subsection{Linear representations of braid groups}
In this section, we construct a linear, local and non-homogeneous representation of $B_n$ and prove that it is equivalent to the well-known Burau representation.
A linear representation $\varphi: B_n \to GL_n(R)$ is called {\it local} if
$$
\varphi(\sigma_i)= I^{i-1} \oplus M_i \oplus I^{n-i-1},\textrm{ for } i \in \lbrace 1,2, \ldots, n-1 \rbrace,
$$
where $I^k$ is the $k \times k$ identity matrix and $M_i$ is a $2 \times 2$ matrix, $i=1,2, \ldots, n-1$, with entries from an integral domain $R$. If $M_1=M_2= \cdots=M_{n-1}$, then $\varphi: B_n \to GL_n(R)$ is called a {\it homogeneous} representation. A linear, local and homogeneous representation of
$S_n$ is defined similarly, where $S_n$ is the symmetric group of degree $n$. A linear representation $\varphi : VB_n \to GL_n(R)$ is called {\it local $($respectively homogeneous$)$} if its restrictions to $B_n$ and $S_n$ are
local (respectively homogeneous).
\medskip

\begin{proposition}\label{P: Proposition}
The map $\varphi : B_n \to GL_n(\mathbb{Z}[t^{\pm 1}, t_1^{\pm 1}, t_2^{\pm 1}, \ldots, t_{n-1}^{\pm 1}])$ defined on the generators by
$$
\varphi(\sigma_i) = I^{i-1} \oplus \left(
\begin{array}{cc}
  1-t &  t t_i\\
t_i^{-1} & 0 \\
\end{array}
\right) \oplus I^{n-i-1}, \text{ for } i = 1, 2, \ldots, n-1,
$$
is a representation of $B_n$. In particular, if $t_i = 1$ for every $i = 1, 2, \ldots, n-1$, then it is the Burau representation. Moreover, $\varphi$ is equivalent to the Burau representation. 
\end{proposition}


\begin{proof}
The fact that $\varphi$ is a representation can be easily deduced. Let us now consider $\phi$ to be the automorphism of the free module $V$ with the basis $\{e_1, e_2, \ldots, e_n\}$ over the ring $R = \mathbb{Z}[t^{\pm 1}, t_1^{\pm 1}, t_2^{\pm 1}, \ldots, t_{n-1}^{\pm 1}]$ which is defined on the basis as follows
$$
\phi :  \left\{
\begin{array}{l}
e_1 \to    e_{1},    \\
e_2 \to t_1 e_2,  \\
\vdots  \\
e_n \to t_1 t_2 \ldots t_{n-1} e_n.
\end{array}
\right.
$$
Next, we consider a representation $\tilde{\varphi}$ defined as $\tilde{\varphi}(\beta) := \phi\varphi(\beta) \phi^{-1}$, where $\beta \in B_n$. Then
$$
\tilde{\varphi} (\sigma_i) = \phi \varphi(\sigma_i) \phi^{-1}  =
 \left\{
\begin{array}{ll}
e_i \to  (1-t) e_i +  e_{i+1}, &   \\
e_{i+1} \to t   e_i, & \\
e_j \to e_j, \text{ for } j \not= i, i+1,
\end{array}
\right.
$$
Hence, $\tilde{\varphi}$ is the Burau representation.

\end{proof}

It is well known that the Burau representation is not faithful for $n>4$.  Bigelow \cite{Bigelow-1} proved the existence of non-trivial elements $b_1 \in B_5$ and $b_2 \in B_6$ in the kernel of the Burau representation (see, for example, \cite{BF}). These elements are
$$
b_1 = [c_1^{-1} \sigma_4 c_1, c_2^{-1} \sigma_4 \sigma_3 \sigma_2 \sigma_1^2 \sigma_2 \sigma_3 \sigma_4 c_2],~~~
b_2 = [d_1^{-1} \sigma_3 d_1, d_2^{-1}  \sigma_3 d_2],
$$
where
$$
c_1 = \sigma_3^{-1}   \sigma_2  \sigma_1^2  \sigma_2  \sigma_4^3  \sigma_3  \sigma_2, ~~~c_2 = \sigma_4^{-1} \sigma_3  \sigma_2  \sigma_1^{-2}  \sigma_2  \sigma_1^2  \sigma_2^2  \sigma_1  \sigma_4^5, ~~~
d_1 =  \sigma_4  \sigma_5^{-1}  \sigma_2^{-1}  \sigma_1,~~~d_2 =  \sigma_4^{-1}  \sigma_5^2  \sigma_2  \sigma_1^{-2}.
$$
This means that the Burau representation is a linear, local and homogenous representation which is not faithful. Furthermore, from the result \cite[Theorem 5.3]{BF}, it follows that there does not exist any linear, local, homogeneous and faithful representation of $B_n$ for $n > 4$. This naturally leads us to ask the following question. 

\begin{question}
Does there exist a linear, local and faithful representation of $B_n$ for $n > 4$?
\end{question}

\subsection{Linear representations of virtual braid groups}\label{SS: Linear representations of the virtual braid group}

Bartholomew and Fenn \cite[Section 7]{BF-2} considered a linear, local and homogeneous representation $\varphi:VB_n \to GL_n(\mathbb{Z}[t^{\pm 1}, \lambda^{\pm 1}])$ defined on the generators as
\begin{align*}
&\sigma_i \mapsto I^{i-1} \oplus \begin{pmatrix}
1-t & \lambda^{-1} t \\
\lambda & 0
\end{pmatrix} \oplus I^{n-i-1},\\
&\rho_i \mapsto I^{i-1} \oplus \begin{pmatrix}
0  & 1 \\
1 & 0
\end{pmatrix} \oplus I^{n-i-1}.
\end{align*}

Clearly, the representation $\varphi: VB_n \to GL_n(\mathbb{Z}[t^{\pm 1}, \lambda^{\pm 1}])$ is virtually symmetric. We now construct a linear, local and non-homogeneous representation of $VB_n$ whose proof is immediate. It suffices to check that the map satisfies the defining relations of $VB_n$.

\begin{proposition}\label{not-homogeneous-virtual}
The map $\psi : VB_n \to GL_n(\mathbb{Z}[t^{\pm 1}, \lambda^{\pm 1}, t_1^{\pm 1}, t_2^{\pm 1}, \ldots, t_{n-1}^{\pm 1}])$ defined on the generators by
$$\psi(\sigma_i) = I^{i-1} \oplus \left(
\begin{array}{cc}
  1-t &    t t_i \lambda^{-1}\\
\lambda t_i^{-1} & 0 \\
\end{array}
\right) \oplus I^{n-i-1},$$
$$
\psi(\rho_i) = I^{i-1} \oplus \left(
\begin{array}{cc}
0 &   t_i \\
 t_i ^{-1} & 0 \\
\end{array}
\right) \oplus I^{n-i-1}, \text{ for }i = 1, 2, \ldots, n-1,
$$
is a representation of $VB_n$.
\end{proposition}

Let us consider the ring $R=\mathbb{Z}[t^{\pm 1}, \lambda^{\pm 1}, t_1^{\pm 1}, t_2^{\pm 1}, \ldots, t_{n-1}^{\pm 1}]$ and a free $R$-module $V$ with basis $\{e_1, e_2, \ldots, e_n\}$. Then we can rewrite the above representation $\psi: VB_n \to GL(V)$ as

\begin{align*}
\psi(\sigma_i) =&  \left\{
\begin{array}{ll}
e_i \to  (1-t) e_i + \lambda t_i^{-1} e_{i+1}, &   \\
e_{i+1} \to t t_i \lambda^{-1} e_i, & \\
e_j \to e_j, \text{ for } j \not= i, i+1,
\end{array}
\right.\\
\psi(\rho_i) =&  \left\{
\begin{array}{ll}
e_i \to   t_i^{-1} e_{i+1}, &   \\
e_{i+1} \to t_i e_i, & \\
e_j \to e_j, \text{ for } j \not= i, i+1.
\end{array}
\right.
\end{align*}

The following proposition generalises the result from \cite[Theorem 7.1, part 3]{BF-2}.

\begin{proposition}\label{P: Proposition2}
The representation $\psi$ is equivalent to a virtually symmetric representation which is local and homogeneous.
\end{proposition}

\begin{proof}
Let $\theta$ be the automorphism of $V$, which is defined on the basis by
$$
\theta =  \left\{
\begin{array}{l}
e_1 \to    e_{1},    \\
e_2 \to t_1 e_2,  \\
\vdots  \\
e_n \to t_1 t_2 \ldots t_{n-1} e_n.
\end{array}
\right.
$$
Consider the representation $\tilde{\psi} = \theta \psi \theta^{-1}$. By definition, this representation is equivalent to $\psi$ and is a virtually symmetric representation. Indeed,
$$
\tilde{\psi}(\sigma_i) =  \left\{
\begin{array}{ll}
e_i \to  (1-t) e_i + \lambda  e_{i+1}, &   \\
e_{i+1} \to t  \lambda^{-1} e_i, & \\
e_j \to e_j, \text{ for } j \not= i, i+1,
\end{array}
\right.
$$
$$
\tilde{\psi}(\rho_i) =  \left\{
\begin{array}{ll}
e_i \to    e_{i+1}, &   \\
e_{i+1} \to  e_i, & \\
e_j \to e_j, \text{ for } j \not= i, i+1.
\end{array}
\right.
$$
\end{proof}

\section{Virtual link groups}\label{S: Virtual link groups}
In the proof of Proposition \ref{P: main-representation-in-file} it is shown that the representations $\varphi_M$ and $\varphi_S$ are equivalent, and $\varphi_S$ is virtually symmetric. In this section, we use these representations to associate a group to each virtual link by various approaches.
\par

\subsection{Braid approach}\label{SS: The braid approach}

It is known \cite{BB-1, BMN-1} that for a given representation of $VB_n$ into the automorphism group of some group or module, one can assign a group to any braid $\beta \in VB_n$. Let $\varphi:VB_n \to \Aut(H)$ be a representation of $VB_n$ into the automorphism group of some group $H=\langle h_1,h_2,\ldots, h_m~~|~~\mathcal{R}\rangle$, where $\mathcal{R}$ is the set of defining relations. For a given $\beta \in VB_n$, we associate the group

$$
G_{\varphi}(\beta)=\langle h_1, h_2, \ldots, h_m~~|~~\mathcal{R}, h_i=\varphi(\beta)(h_i), i=1,2,\ldots,m \rangle.
$$
\par
\vspace*{.5cm}
For each $\beta \in VB_n$, let  $G_M(\beta)$ and $G_S(\beta)$ be the groups corresponding to the representations $\varphi_M$ and $\varphi_S$, respectively.
\medskip

The following result can be proved on similar lines as done in\cite[Section 6]{BMN-1}.
\begin{theorem}
If $\beta$ and $\beta'$ are two virtual braids such that their closures define the same virtual link, then $G_M(\beta) \cong G_M(\beta').$
\end{theorem}

The above theorem implies that the group $G_M(\beta)$ is an invariant of virtual links.

\begin{theorem}\label{isomorphism-of-G_M-and-G_S}
Let $\beta \in VB_n$. Then the group $G_M(\beta)$ is isomorphic to the group $G_S(\beta)$. In particular, $G_S(\beta)$ is a link invariant.
\end{theorem}
\begin{proof}
For  $\beta \in VB_n$, the group $G_S(\beta)$ has a following presentation
$$
G_S(\beta)=\langle x_1, \ldots, x_n, v_1, \ldots, v_n~~|~~[v_i, v_j]=1, \varphi_S(\beta)(x_i)=x_i, \varphi_S(\beta)(v_j)=v_j \rangle.
$$
Consider the map $\phi: F_{n,n} \to F_{n,n}$ defined in Proposition \ref{P: main-representation-in-file}.
So we have
\begin{align*}
G_S(\beta) &\cong \langle x_1, \ldots, x_n, v_1, \ldots, v_n~~|~~[v_i, v_j]=1, \phi \big(\varphi_S(\beta)(x_i)\big)=\phi(x_i), \phi\big(\varphi_S(\beta)(v_j)\big)=v_j \rangle\\
\\
&\cong \langle x_1, \ldots, x_n, v_1, \ldots, v_n~~|~~[v_i, v_j]=1, \varphi_M(\beta)\big(\phi(x_i)\big)=\phi(x_i), \varphi_M(\beta)\big(\phi(v_j)\big)=v_j \rangle\\
\\
&\cong \langle x_1, \ldots, x_n, v_1, \ldots, v_n~~|~~[v_i, v_j]=1, \big(\varphi_M(\beta)(x_i)\big)^{v_i\ldots v_n}=x_i^{v_i \ldots v_n}, \varphi_M(\beta)\big(v_j\big)=v_j \rangle\\
\\
&\cong \langle x_1, \ldots, x_n, v_1, \ldots, v_n~~|~~[v_i, v_j]=1, \varphi_M(\beta)(x_i)=x_i, \varphi_M(\beta)\big(v_j\big)=v_j \rangle\\
\\
&\cong G_M(\beta).
\end{align*}
\end{proof}

\begin{remark}
Let $D_L$ be a link diagram representing the virtual link $L$ and $\beta$ be a virtual braid whose closure is equivalent to $L$. Then one can write a presentation of the group $G_S(\beta)$ using the diagram $D_L$. For more on this approach, we refer the reader to \cite[Subsection 6.2]{BMN-1}.
\end{remark}

\begin{notation}
Henceforth, the notation $G_S(L)$ stands for the group $G_S(\beta)$, where $\beta$ is a virtual braid whose closure is the virtual link $L$.
\end{notation}

Hereafter, for a given virtual link $L$, the term {\it virtual link group of $L$} stands for the group $G_S(L)$.

\begin{remark}\label{R: one component link}
	We note that if the closure of $\beta \in VB_n$ is a virtual knot, then $G_S(\beta) \cong G_{\tilde{\varphi}_A}(\beta)$, where $\tilde{\varphi}_A$ is the generalised Artin representation considered in Subsection \ref{SS: generalised Artin representation}.
\end{remark}

\begin{remark}
	If we put the relations $v_i=1$ for all $i$ in the presentation of $G_S(L)$, then we recover the group $G_0(L)$ defined by Kauffman \cite{Kauffman-1} which corresponds to the representation $\varphi_0$ mentioned in Section \ref{S: Virtually symmetric representations}.
\end{remark}

\subsection{Gauss diagram approach}\label{GaussDiagramsAndGroups}
A \textit{Gauss diagram} consists of a finite number of circles oriented anticlockwise with a finite number of signed arrows whose heads and tails lie on the circles. If the head and tail of an arrow lie on the same circle, then it is said to be a {\it chord}. For every oriented virtual link diagram one can construct a Gauss diagram. There is a one-to-one correspondence between virtual links and Gauss diagrams considered up to the moves shown in Figure \ref{fig:GaussMoves}. For more details, see \cite{GPV-1, Kauffman-1, Petter-1}.

\begin{figure}[H]
	\centering
	\includegraphics[scale=1.4]{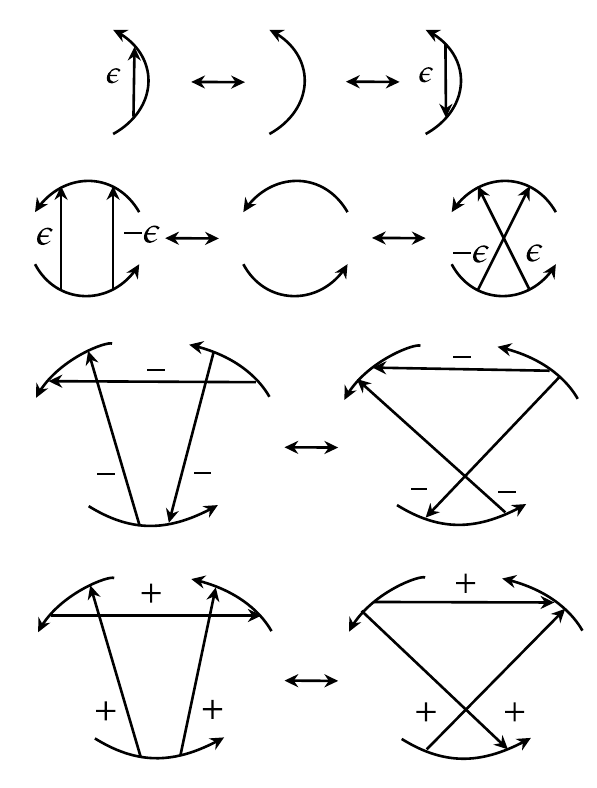}
	\caption{Moves on Gauss diagrams corresponding to the Reidemeister moves.}
	\label{fig:GaussMoves}
\end{figure}
\par

The advantage of using the virtually symmetric representation $\varphi_S:VB_n \to \Aut(F_{n,n})$ over its equivalent representation $\varphi_M: VB_n \to \Aut(F_{n,n})$ is that for a given virtual link $L$, we can write a presentation of the virtual link group $G_S(L)$ using the corresponding \textrm{Gauss diagram}. Let $D$ be a Gauss diagram with $m$ circles representing the virtual link $L$. Label the circles with symbols $v_1,\ldots,v_m$. If we cut the circles at the head and tail of each arrow, then the circles of $D$ are divided into arcs to which we assign symbols $x_1, x_2, \ldots, x_n$. Next, we define a group
$$\pi_D=\langle x_1, \ldots, x_n, v_1, \ldots, v_m~|~~\mathcal{R},v_iv_j=v_jv_i~\textrm{where} ~1 \leq i,j \leq m\rangle,$$
where $\mathcal{R}$ is the set of relations defined for each signed arrow as depicted in Figure \ref{Gauss-Diagram-Relations}.

\begin{figure}[H]
	\centering
	\includegraphics[scale=1.4]{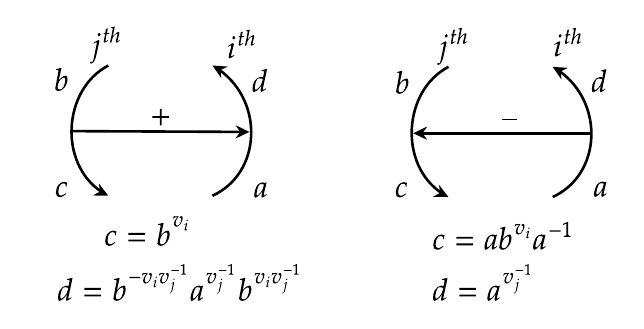}
	\caption{Relations for the group $\pi_D$.}
	\label{Gauss-Diagram-Relations}
\end{figure}

We note down the following result which is not difficult to prove.
\begin{proposition}
If $D$ is a Gauss diagram representing virtual link $L$, then $\pi_D \cong G_S(L)$.
\end{proposition}
\section{Marked gauss diagrams}\label{S: marked Gauss Diagrams}
This is one of the main sections in the paper. In this section, we define and study the Gauss diagrams with an additional structure, and extend the notion of the virtual link group to these diagrams using the similar approach as used in Subsection \ref{GaussDiagramsAndGroups}.

\begin{definition}
A \textit{marked Gauss diagram} is a collection of a finite number of circles oriented anticlockwise having a finite number of signed arrows whose heads and tails lie on the circles along with a finite number of signed nodes on the circles which do not touch any of the arrows. If the head and tail of an arrow lie on the same circle, then it is said to be a {\it chord}.
\end{definition}

By a {\it $1$-circle marked Gauss diagram}, we mean a marked Gauss diagram having only one circle. Figures \ref{F: Example-1} and \ref{F: Example-2} illustrate some examples of marked Gauss diagrams.

\begin{figure}[H]
	\centering
	\includegraphics[scale=1.4]{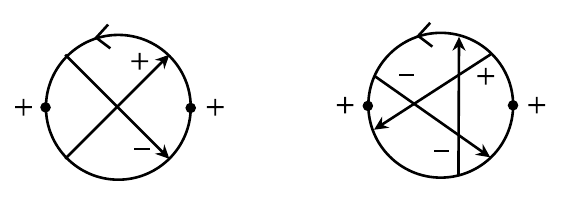}
	\caption{Examples of $1$-circle marked Gauss diagrams.}
	\label{F: Example-1}
\end{figure}

\begin{figure}[H]
	\centering
	\includegraphics[scale=1.4]{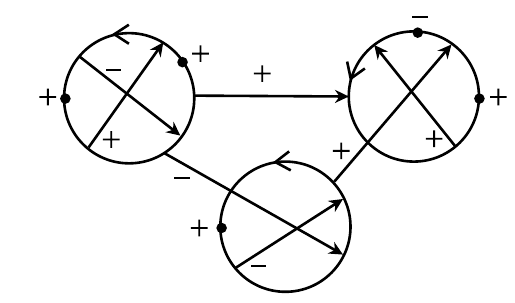}
	\caption{A marked Gauss diagram with three components.}
	\label{F: Example-2}
\end{figure}

We consider marked Gauss diagrams up to the equivalence relation generated by finite sequence of the moves shown in figures \ref{fig:GaussMoves} and \ref{F: AddGaussMoves}, and we call these moves as the \textit{marked Reidemeister moves}. Note that in Figure \ref{F: AddGaussMoves}, the moves involve only those arcs which belong to the same circle in a given marked Gauss diagram.

\begin{figure}[H]
	\includegraphics[width=9cm]{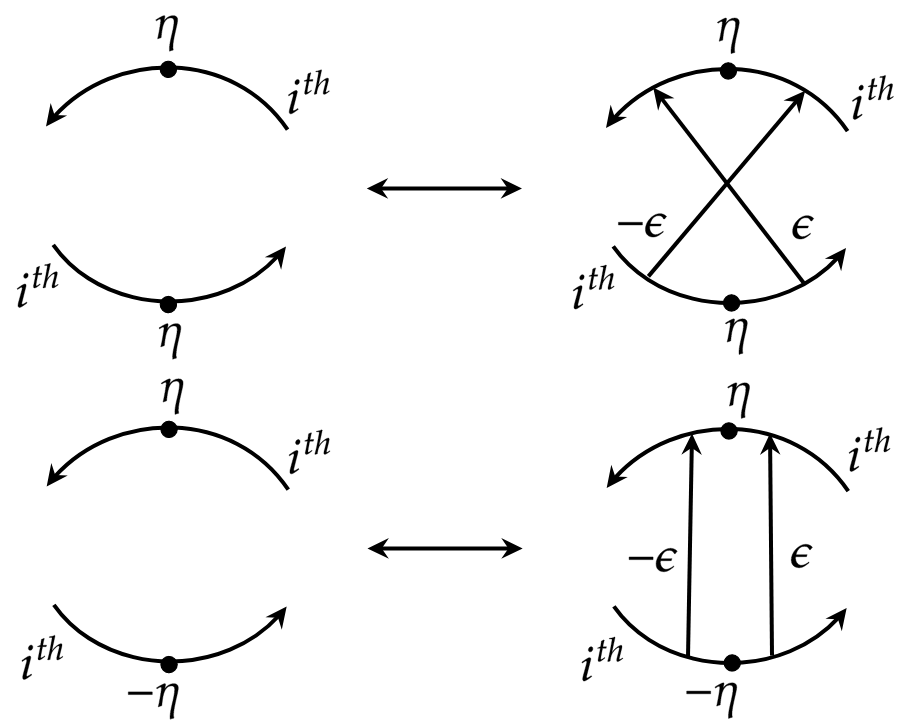}
	\caption{Additional moves on marked Gauss diagrams.}
	\label{F: AddGaussMoves}
\end{figure}

It is clear from Figure \ref{F: AddGaussMoves} that the concept of marked Gauss diagrams is a proper generalisation of Gauss diagrams, that is, there is a canonical injective map from the set of equivalence classes of Gauss diagrams to the set of equivalence classes of marked Gauss diagrams.
\par
\vspace*{.2cm}
Next, to each marked Gauss diagram we associate a group as follows. Let $D$ be a marked Gauss diagram with $m$ circles where the $i^{th}$ circle is labelled as $v_i$. Then we cut the circles at the head and tail of each arrow and at each node point dividing the circles of $D$ into arcs to which we assign symbols $x_1, x_2, \ldots, x_n$. Define a group
$$
\Pi_D:= \langle x_1, \ldots, x_n, v_1, \ldots, v_m~|~\mathcal{R},v_iv_j=v_jv_i~\textrm{for}~1 \leq i, \leq j \leq m \rangle,
$$
where $\mathcal{R}$ consists of relations for each arrow and each node in $D$ as shown in Figure \ref{Marked-Gauss-Diagram-Relations}. It is easy to check that the group $\Pi_D$ is invariant under the marked Reidemeister moves. We note that if $D$ is a Gauss diagram, then $\Pi_D \cong \pi_D$ which shows that the notion of the virtual link group can be extended to marked Gauss diagrams.

\begin{figure}[H]
    \includegraphics[width=9cm]{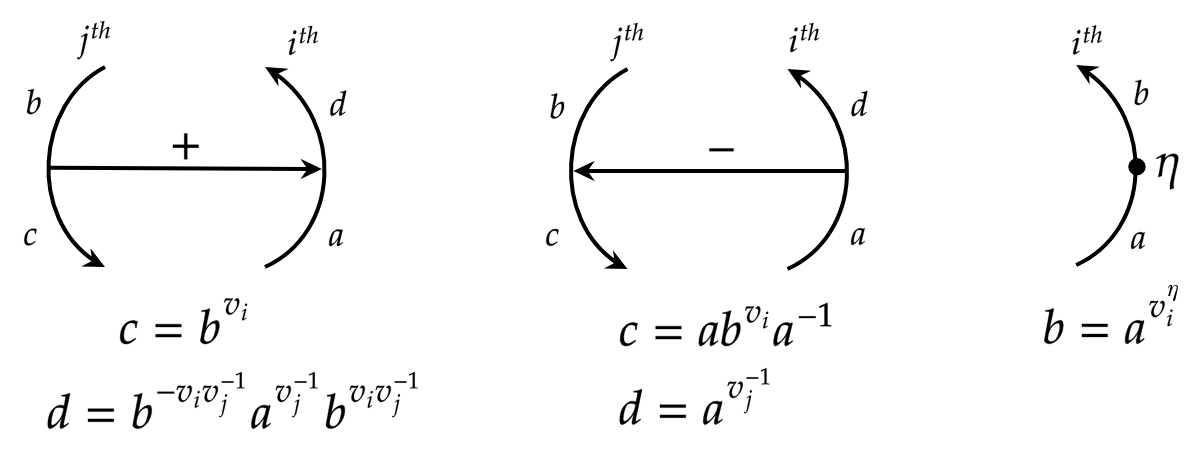}
    \caption{Relations for the group $\Pi_D$.}
    \label{Marked-Gauss-Diagram-Relations}
\end{figure}
\par

\begin{proposition}
The number of nodes, and the sum and product of sign of nodes in a given marked Gauss diagram are invariant under the marked Reidemeister moves.
\end{proposition}
\begin{proof}
Recall that the marked Reidemeister moves consist of moves shown in figures \ref{fig:GaussMoves} and \ref{F: AddGaussMoves}. Since the moves in Figure \ref{fig:GaussMoves} do not have any nodes and the moves shown in Figure \ref{F: AddGaussMoves} affect neither the number of nodes nor the signs of the nodes in a given marked Gauss diagram. Hence, the proof follows.
\end{proof}

\begin{figure}[H]
	\centering
	\includegraphics[scale=1.4]{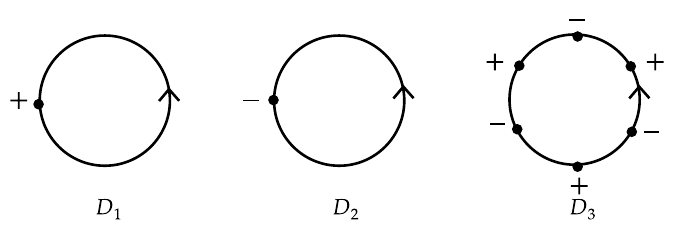}
	\caption{Non-equivalent marked Gauss diagrams.}
	\label{fig:general-Gauss-diagrams-with-isomorphic-group}
\end{figure}

\begin{example}
Consider the diagrams in Figure \ref{fig:general-Gauss-diagrams-with-isomorphic-group}. Note that the group $\Pi_{D_1}$ has presentation $\langle a,v ~~|~~a=a^{v^{-1}} \rangle$
and $\Pi_{D_2}$ has presentation $\langle b,v ~~|~~b=b^{v} \rangle$. Clearly $\Pi_{D_1} \cong \Pi_{D_2} \cong \mathbb{Z}^2$ but diagrams $D_1$ and $D_2$ are not equivalent since the sum of signs of the nodes are not equal. Furthermore, in $D_3$ there are equal number of positive and negative nodes and no chords, and hence $\Pi_{D_3} \cong \mathbb{Z}*\mathbb{Z}$.
\end{example}

\section{Marked virtual link diagrams}\label{S: Marked virtual link diagrams}
In this section, we give an interpretation of marked Gauss diagrams in terms of planar diagrams. 
\medskip

Let $G=(V,E)$ be a directed graph, where $V$ denotes the set of vertices and $E$ denotes the set of directed edges. A {\it diwalk}
is an alternating sequence of vertices and edges $v_0, e_1, v_1, \ldots, v_{n-1}, e_n, v_n$ with edge $e_i$ directed from $v_{i-1} $ to $v_i$, for every  $v_i \in V$ and $e_i \in E$ . A {\it directed cycle} is a diwalk in which all the vertices except the first and last are different. From now onwards, by a {\it cycle} we mean a directed cycle.
\medskip

Beineke and Harary  \cite{BH-1} defined a {\it marked graph} as a directed graph in which each vertex is assigned either positive or negative sign. We define {\it marked cycles} as a marked graph consisting only of cycles and no two cycles share a common vertex. For example, see Figure \ref{F: Illustration of marked cycles}.

\begin{figure}[H]
	\centering
	\includegraphics[scale=1.4]{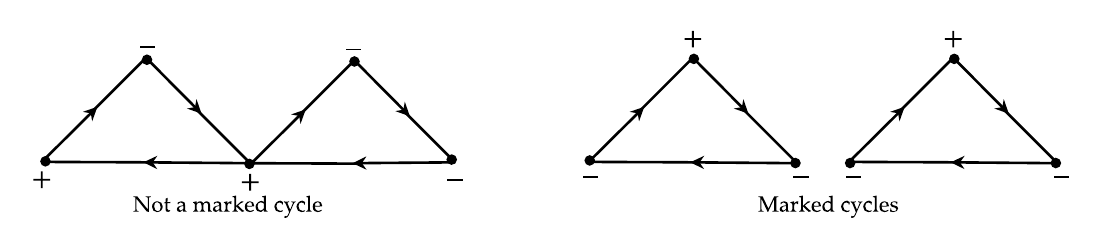}
	\caption{Illustration of marked cycles.}
	\label{F: Illustration of marked cycles}
\end{figure}

Further, Fleming and Mellor \cite{FM-1} defined a {\it virtual spatial graph diagram} as a generic immersion of a directed graph in $\mathbb{R}^2$, where each double point is either a classical crossing or a virtual crossing. Analogously, we can define the following.

\begin{definition}
A {\it marked virtual link diagram} is a generic immersion of marked cycles in $\mathbb{R}^2$ with the information of virtual and classical crossings at double points. If it is a one component diagram, then it is said to be a {\it marked virtual knot diagram}.
\end{definition}

Note that for any given marked Gauss diagram, we can draw a marked virtual link diagram, and the converse also holds. Please refer to  Figure \ref{F: MGDandMVLD} for an illustration.

\begin{figure}[H]
	\centering
	\includegraphics[scale=1.2]{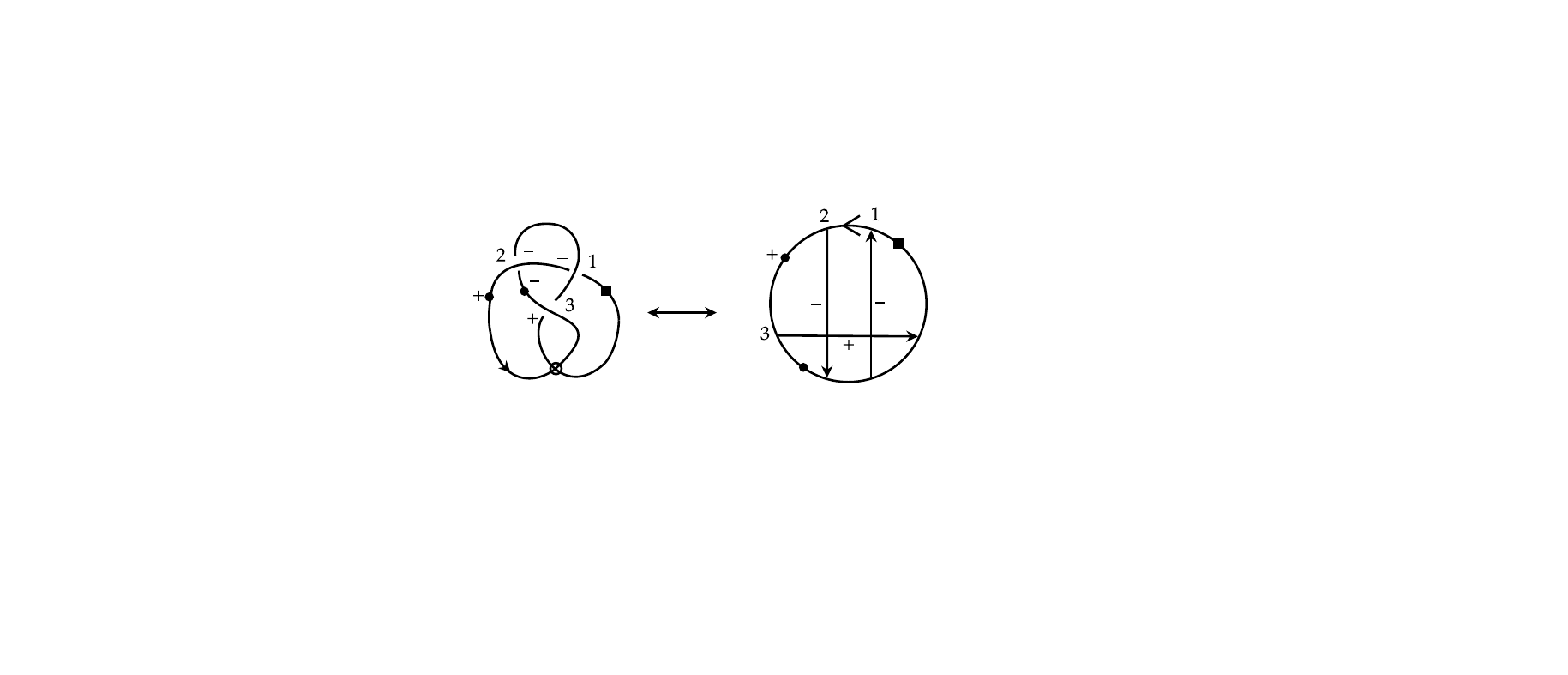}
	\caption{Marked virtual knot diagram and its marked Gauss diagram.} \label{F: MGDandMVLD}
\end{figure}
We say that two marked virtual link diagrams are {\it equivalent} if they are related by a finite sequence of moves shown in Figure \ref{F: MarkedReidemeisterMoves}, and are called as the {\it marked Reidemeister moves}. 
\begin{figure}[H]
	\centering
	\includegraphics[scale=1.2]{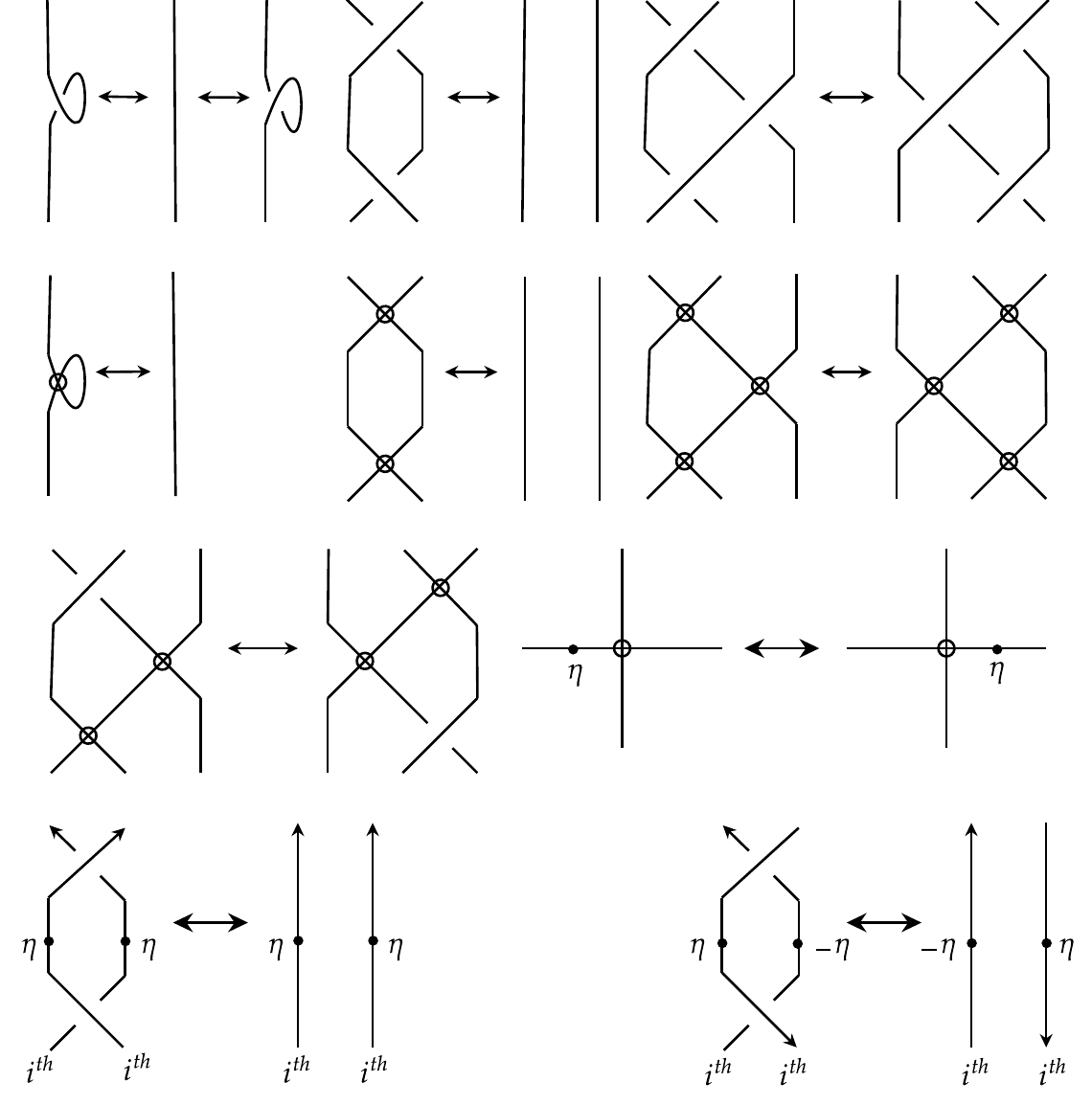}
	\caption{Marked Reidemeister moves.} 
	\label{F: MarkedReidemeisterMoves}
\end{figure}
It is clear that there is a one-to-one correspondence between the set of equivalence classes of marked Gauss diagrams and the set of equivalence classes of marked virtual link diagrams. We note that the moves shown in Figure \ref{F: ForbiddenReidMoves} are forbidden and can not be obtained from the moves shown in Figure \ref{F: MarkedReidemeisterMoves}.
\begin{figure}[H]
	\includegraphics[scale=1]{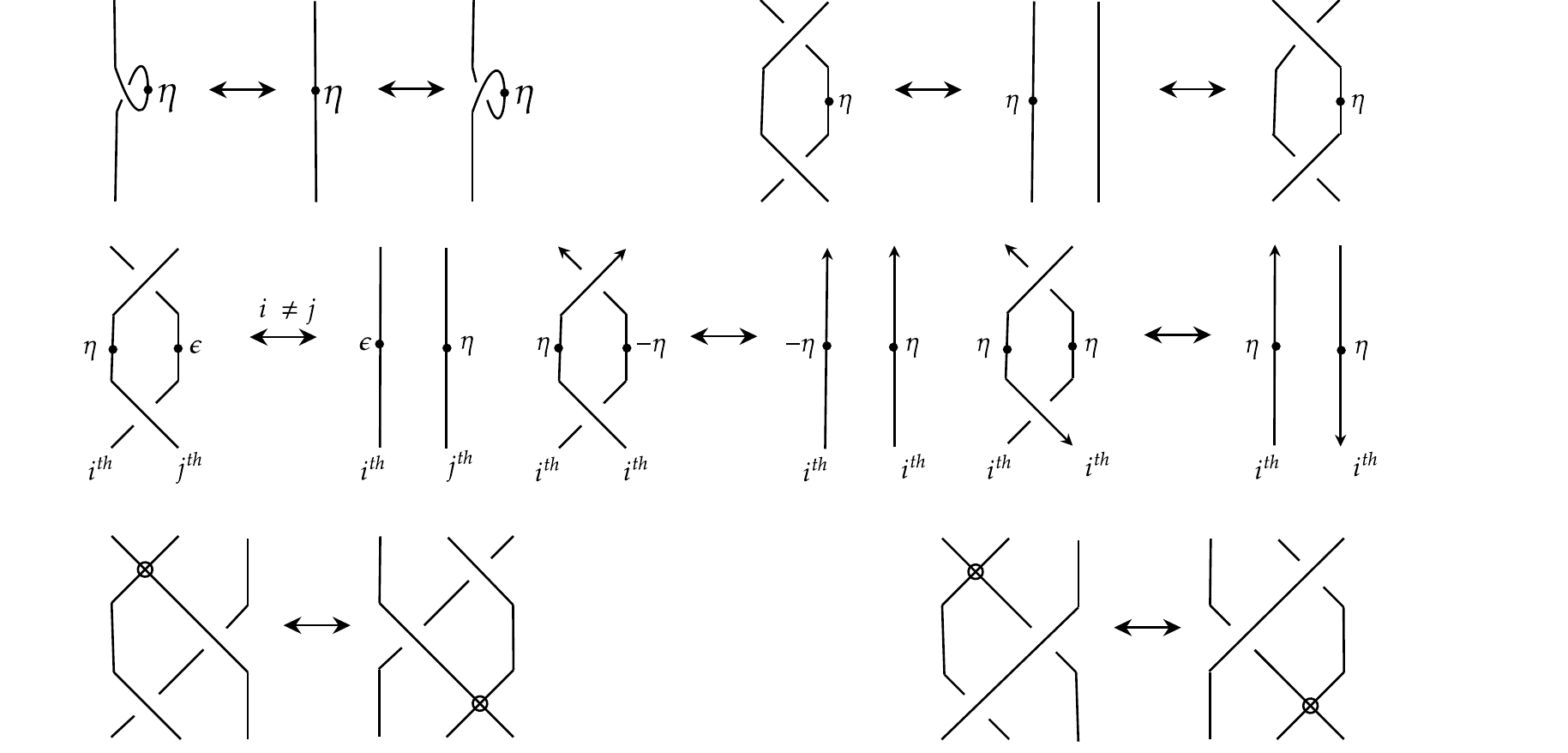}
	\caption{Forbidden Reidemeister moves on marked virtual link diagrams.} \label{F: ForbiddenReidMoves}
\end{figure}

Let $D$ be a marked Gauss diagram. Then we can write a presentation of the group $\Pi_D$ from a marked virtual link diagram representing $D$ as follows. Let $D_L$ be a marked virtual link diagram with $m$ components representing $D$. We begin by enumerating all the components of $D_L$ with numbers from $1$ to $m$ and label the $i^{th}$ component with $v_i$. After this, we divide the diagram into arcs: from one classical crossing to the next classical crossing, from one classical crossing to the next node point, from one node point to the next classical crossing and from one node point to the next node point. Label the arcs as $x_1,x_2,\ldots, x_n$. Then the group $\Pi_D$ has a presentation with generators $x_1, x_2,\ldots,x_n, v_1,v_2, \ldots, v_m$ and defining relations at crossings and nodes as shown in Figure \ref{F: Marked Crossings Relations}, along with the commutativity of all $v_i$'s $(1 \leq i \leq m)$ with each other. Note that there are no relations at virtual crossings.
\begin{figure}[H]
\includegraphics[width=11cm]{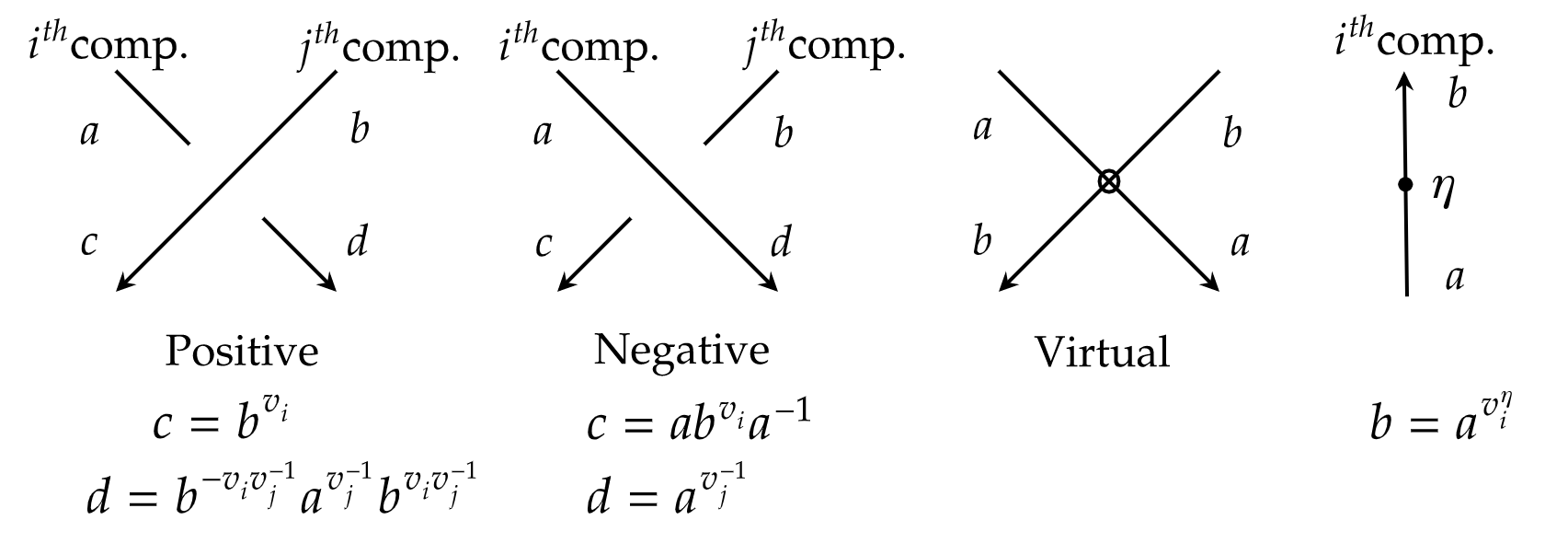}
\caption{Defining relations obtained from crossings and nodes of a marked virtual link diagram.}\label{F: Marked Crossings Relations}
\end{figure}

\section{Realization of irreducible $C_1$-groups}\label{S: Realization-of-irreducible $C_1$-group}
In this section, we prove that every $1$-irreducible $C_1$-presentation of deficiency $1$ or $2$ is the group of some $1$-circle marked Gauss diagram. We first recall the definition of $C$-groups given by Kulikov \cite{Kulikov-1, Kulikov-2}. A group $G$ is called a $C$-{\it group} if it admits a presentation $\langle X~~|~~\mathcal{R} \rangle$, where $X=\{x_1, x_2, \ldots, x_n\}$ and relations $\mathcal{R}$ are of the type $$w_{i,j}^{-1} x_i w_{i,j} = x_j$$ for some $x_i, x_j  \in X$ and some words $w_{i,j}$ in $X^{\pm 1}$. Such a presentation is known as a $C$-{\it presentation}. Gilbert and Howie \cite{GH-1} called these groups as LOG groups. It is established in  \cite{Kulikov-2} that  every $C$-group can be realized as the fundamental group of complement of some $n$-dimensional $(n \geq 2)$ compact orientable manifold without boundary embedded in $\mathbb{S}^{n+2}$. In particular, any classical link group is a $C$-group.
\medskip

We now define $C_m$-groups which are a specific type of $C$-groups. For a non-negative integer $m$, a group $G$ is called a $C_m$-{\it group} if it can be defined by a set of generators $Y = X \sqcup V_m$, where
$X = \{ x_1, x_2, \ldots, x_n\}$, $V_m = \{ v_1, v_2, \ldots, v_m  \}$ and a set of relations $\mathcal{R}$ given by
\begin{align*}
	w_{i,j}^{-1} x_i w_{i,j} &= x_j,~~\mbox{for some}~x_i, x_j  \in X ~\mbox{and some words}~~w_{i,j}~\mbox{in}~Y^{\pm 1},\\
	v_i v_j &= v_j v_i,~~\mbox{for all}~v_i, v_j  \in V_m.
\end{align*}
We call the presentation $\langle Y~~|~~\mathcal{R} \rangle$ as a $C_m$-{\it presentation}. 
\medskip

Notice that all $C_m$-groups are $C_{m-1}$-groups for $m \geq 1$. In particular, all $C_m$-groups are $C$-groups. It is easy to see that the abelianization of a $C_m$-group is a free abelian group. If we put $x_i=1$ for all $i=1,2, \ldots,n$, then we get the free abelian group of rank $m$ whereas, if we put $v_j=1$ for all $j=1,2, \ldots, m$, then we get a $C$-group.
\medskip

A $C$-group is called {\it irreducible} if all generators in its $C$-presentation are conjugates of each other. Analogously, we say that a finitely generated $C_1$-group is {\it $1$-irreducible} if all $x_i \in X$ in a $C_1$-presentation $\langle Y~~|~~\mathcal{R} \rangle$, where $Y=X \sqcup V_1$,  are conjugate to each other. Equivalently, its abelianization is the free abelian group of rank $2$. Next we extend this definition to $C_m$-groups for $m \geq 1$ as follows. 
\medskip

Associate a graph (not directed) $\Gamma_m$ to a $C_m$-presentation $\langle Y~~|~~\mathcal{R} \rangle$ with vertices $x_1, x_2, \ldots, x_n$ and edges $e^{j}_{i}$ between vertices $x_i$ and $x_j$ if there is a relation $w_{i,j}^{-1} x_iw_{i,j}=x_j$ in $\mathcal{R}$. We then say that the $C_m$-presentation $\langle Y~~|~~R \rangle$ with $m\geq 1$ is {\it $m$-irreducible} if the associated graph $\Gamma_m$ has $m$-connected components and the corresponding $C_m$-group is called an {\it $m$-irreducible} group. Equivalently, its abelianization is of rank $2m$. Note that in this case $n \geq m$.

\begin{remark}
	Note that an $m$-irreducible $C_m$-group $(m>0)$ is not an $(m-i)$-irreducible $C_{m-i}$-group, where $1 \leq i \leq m$.
\end{remark}
\medskip

The {\it deficiency} of a group presentation is the number of generators minus the number of relations. The {\it deficiency} of a finitely presented group $G$ is defined as the maximum deficiency of finite group presentations for $G$. It is easy to see that the group associated to a marked Gauss diagram having $m$ number of components is an $m$-irreducible $C_m$-group. The following result is about the deficiency of the groups associated to $1$-circle marked Gauss diagrams.
\medskip

\begin{proposition}\label{P: one-way}
	The group associated to a given $1$-circle marked Gauss diagram is a $1$-irreducible $C_1$-group of deficiency $1$ or $2$ and its second integral homology group is cyclic.
\end{proposition}

\begin{proof}
	It is easy to see that the group associated to a $1$-circle marked Gauss diagram is a $1$-irreducible $C_1$-group with deficiency greater than or equal to $1$. Let $D$ be a $1$-circle marked Gauss diagram and $\Pi_D$ the group associated to $D$ having deficiency $d$. Then the group $\Pi_D$ has a presentation $\mathcal{P}$ with $n+d$ generators $x_i$ and $n$ relations $r_j$. Let $X$ be the
	$2$-skeleton of the Eilenberg-MacLane space $K(\Pi_D,1)$. To be precise, $X$ is obtained by gluing $n$ many $2$-disks to the one-point union of $n+d$ circles along the relations $r_j$. By construction, $\pi_1(X) \cong \Pi_D$. Then the cellular chain complex of $X$ is
	
	$$\ldots\to 0\to \mathbb{Z} ^n \stackrel{\partial_2}{\to} \mathbb{Z}^{n+d}
	\stackrel{\partial_1}{\to}\mathbb{Z},$$
	where $\partial_1$ is the zero map.	Since $\Pi_D$ is a $1$-irreducible $C_1$-group, $H_1(X)\cong \mathbb{Z}\oplus \mathbb{Z}$. As rank$($coker $\partial_2)+$rank$(\partial_2(\mathbb{Z}^n))=n+d$ and rank$(\partial_2(\mathbb{Z}^n))\leq n$, we have $d \leq 2$. Thus $\Pi_D$ has deficiency either $1$ or $2$.
	Also, it follows that $H_2(X)$ is $0$ for $d=2$ and $\mathbb{Z}$ for $d=1$. Hence, $H_2(\Pi_D) \cong H_2(X)$ is cyclic.
\end{proof}

\begin{corollary}
	Every virtual knot group has deficiency $1$ or $2$, and its second homology group is cyclic.
\end{corollary}

It is well-known that if $K$ is a classical knot, then the fundamental group of its complement $\pi_1(\mathbb{S}^3 - K)$ has deficiency $1$.

\begin{example}
	It follows from Remark \ref{R: one component link}, \cite[Proposition 4]{BB-1} and Theorem \ref{isomorphism-of-G_M-and-G_S} that $G_S(K) \cong \pi_1(\mathbb{S}^3 - K) * \langle v \rangle$, and hence it is of deficiency $2$ and $H_2\big(G_S(K)\big)=0$.
\end{example}

\begin{example}\label{ex:Group-of-general-Gauss-diagram}
	Let $G$ be a group having an irreducible $C$-presentation of deficiency $0$ such that $H_2(G) \neq 0$. Let $G^*=G * \langle v \rangle$ be the free product of $G$ and $\langle v \rangle$. Then $H_2(G^*) \neq 0$ since $H_2(G^*) \cong H_2(G) \oplus H_2(\mathbb{Z}) \cong H_2(G)$. Assuming Theorem \ref{RealizationOfirreducible $C_1$-presentation}, we have that $G^*$ is the group of some marked Gauss diagram, so we get that $G^*$ has deficiency $1$. This illustrates the existence of marked Gauss diagrams whose associated groups have deficiency $1$. Gordon \cite{Gordon-1} gave a family of irreducible $C$-presentations of deficiency $0$ whose second homology groups are $\mathbb{Z}$. Moreover, one can find an irreducible $C$-presentation of deficiency $0$ with second homology group of order $2$ in \cite{BMS-1}.
\end{example}

Let $T$ be a Gauss diagram corresponding to the trivial knot. The following result is of independent interest. 
\begin{proposition}
	There are infinitely many marked Gauss diagrams with associated group isomorphic to $\Pi_T = \mathbb{Z}* \mathbb{Z}$.
\end{proposition}
\begin{proof}
	Let $D$ be a $1$-circle marked Gauss diagram with equal number of positive and negative nodes and having no chords. Then clearly, $\Pi_D \cong \Pi_T = \mathbb{Z} * \mathbb{Z}$.
\end{proof}

The aim of this section is to prove that every $1$-irreducible $C_1$-presentation of deficiency $1$ or $2$ can be realized as the group of some $1$-circle marked Gauss diagram. But before this, we need to define some terms.  
\begin{definition}\label{cyclic irreducible $C_1$-presentation}
A \textit{cyclic irreducible $C_1$-presentation} is a $1$-irreducible $C_1$-presentation of the form
$$
\langle x_1, x_2, \ldots, x_n,v~~|~~r_1, r_2, \ldots , r_n \rangle,
$$
where the relation $r_j$ is of the form $x_{j+1} ^{-1} x_j ^{w_j}$ for $j \in \mathbb{Z}_n$, and  $w_j$ is a word in alphabets $x_i^{\pm 1}$ and $v^{\pm 1}$.
\end{definition}

Now considering a cyclic irreducible $C_1$-presentation defined in Definition \ref{cyclic irreducible $C_1$-presentation}, we define a realizable irreducible $C_1$-presentation as follows.
\begin{definition}\label{realizable irreducible $C_1$-presentation}
A \textit{realizable irreducible $C_1$-presentation} is a cyclic irreducible $C_1$-presentation where each $w_j$ is an element of the set $\{ v^{-\epsilon}x_i^{\epsilon}, v^{\epsilon}~~|~~ i=1, 2, \dots , n \text{ and } \epsilon=\pm 1 \}$ and satisfies the following conditions:
\begin{enumerate}
	\item If $w_k=v^{-1}x_p$ for some $p$, then $w_p=v$ and the word $w_j$ is neither equal to $v^{-\epsilon} x_p^{\epsilon}$ nor $v^{-\epsilon} x_k^{\epsilon}$ for any $j \neq k$.
	\item If $w_k=vx_p^{-1}$ for some $p$, then $w_p=v^{-1}$ and the word $w_j$ is neither equal to $v^{-\epsilon} x_p^{\epsilon}$ nor $v^{-\epsilon} x_k^{\epsilon}$ for any $j \neq k$. 
\end{enumerate}
\end{definition}

The proof of the following result is similar to \cite[Lemma 2]{Kim-1}.

\begin{proposition}\label{irreducible $C_1$-presentation-to-cyclic irreducible $C_1$-presentation}
Any $1$-irreducible $C_1$-presentation of deficiency $1$ or $2$ can be transformed to a cyclic irreducible $C_1$-presentation.
\end{proposition}
\begin{proof}
Let $\mathcal{P}= \langle x_1, \ldots, x_n,v ~~|~~r_1, \ldots, r_m \rangle$ be a $1$-irreducible $C_1$-presentation of deficiency $1$ or $2$, that is either $m=n$ or $n-1$. If $m=n-1$, then we add a relation $r_{m+1}=r_m$ and therefore, we assume $m=n$. Next, we consider the graph $\Gamma$ associated to the presentation $\mathcal{P}$ as defined in Section \ref{S: prelim}. We observe that if $\Gamma$ has two edges $e_j^i$ and $e_k^j$ meeting at the vertex $x_j$, then there are relations of the form $x_i^{-1}x_j^{w_{j,i}}$ and $x_j^{-1} x_k^{w_{k,j}}$ in $\mathcal{P}$. Note that removing relation $x_i^{-1}x_j^{w_{j,i}}$ and adding relation $x_i^{-1}x_k^{w_{k,j}w_{j,i}}$ corresponds to an operation on $\Gamma$ of removing the edge $e_j^i$ and adding edge $e_k^i$. Clearly, this operation does not change the underlying group. Since $\Gamma$ is a connected graph and the number of edges is equal to the  number of vertices, $\Gamma$ has exactly one cycle $C$. Now, if the length $l(C)$ of cycle $C$ is $n$, then $\mathcal{P}$ is cyclic irreducible $C_1$-presentation and if not, then because $\Gamma$ is connected, there is an edge $e_j ^i$ such that $x_i$ is in $C$ and $x_j$ is not in $C$. Using the above operation, we get a graph $\Gamma '$ with cycle $C'$ containing all vertices of $C$ and $x_j$, and $l(C')=l(C)+1$. Thus, after finitely many steps we obtain a graph with cycle of length $n$, which gives the desired result.
\end{proof}

\begin{proposition}\label{cyclic irreducible $C_1$-presentation-to-realizable-irreducible $C_1$-presentation}
Every cyclic irreducible $C_1$-presentation can be transformed to a realizable irreducible $C_1$-presentation.
\end{proposition}
\begin{proof}
Consider a cyclic irreducible $C_1$-presentation $\mathcal{P}$ as defined in Definition \ref{cyclic irreducible $C_1$-presentation}. Then $\mathcal{P}$ can be transformed into a realizable irreducible $C_1$-presentation in the following steps:
\begin{itemize}
\item\label{Step1} \textit{Step $1$:} Make each $w_j$ one of the letter in $\{ x_1^{\pm 1}, \ldots, x_n^{\pm 1}, v^{\pm 1} \}$. For example, let $w_j= x_4 x_3 v^{-1}$ and $r_j= x_{j+1}^{-1}x_j^{w_j}$. Then in $\mathcal{P}$, add two more generators say $x_j'$ and $x_j''$, remove the relation $r_j$ and add three more relations $x_j'^{-1} x_j^{x_4}$, ${x_j''}^{-1} x_j'^{x_3}$ and $x_{j+1}^{-1}{x_j''}^{v^{-1}}$. Thus we get a new cyclic irreducible $C_1$-presentation presenting the same group. Moreover, we can assume that in the new cyclic irreducible $C_1$-presentation there are no relations of the type $r_j=x_{j+1}^{-1} x_j^{x_j ^{\epsilon}}$.

\item \textit{Step $2$:} If $r_j=x_{j+1}^{-1} x_j^{x_k^{\epsilon}}$  $(x_k \neq x_j)$, then remove the relations $r_j$ and $r_k=x_{k+1}^{-1} x_k^{w_k}$, and add three generators $x_{k,1}, x_{k,2}, X_j$, five relations $x_{k,1}=x_k^{v^{\epsilon}}, x_{k,2}=x_{k,1}^{v^{-\epsilon}}, X_j=x_j^{v^{\epsilon}},x_{j+1}=X_j^{v^{-\epsilon}x_k^{\epsilon}}, x_{k+1}=x_{k,2}^{w_k}$ in $\mathcal{P}$. Moreover, replace $x_k$ by $x_{k,2}$ in words $w_i$ in $\mathcal{P}$ for $i \neq j$.
\end{itemize}
\end{proof}

As a consequence, we have the following result.

\begin{corollary}\label{irreducible $C_1$-presentation-to-realizable-irreducible $C_1$-presentation}
Any $1$-irreducible $C_1$-presentation of deficiency $1$ or $2$ can be transformed to a realizable irreducible $C_1$-presentation.
\end{corollary}

The proof of the following theorem is along the same lines as in \cite[Theorem 3]{Kim-1}.

\begin{theorem}\label{RealizationOfirreducible $C_1$-presentation}
Any $1$-irreducible $C_1$-presentation of deficiency $1$ or $2$ can be realized as the group of a marked Gauss diagram.
\end{theorem}
\begin{proof}
Let $\mathcal{P}$ be a $1$-irreducible $C_1$-presentation of deficiency $1$ or $2$. By Corollary \ref{irreducible $C_1$-presentation-to-realizable-irreducible $C_1$-presentation}, we can assume that $\mathcal{P}$ is a realizable irreducible $C_1$-presentation having $n+1$ generators $\{x_i, v ~|~ i= 1, 2, \dots, n\}$ and $n$ relations $r_1, \ldots, r_n$. Next, we consider a circle with anticlockwise orientation and mark $n$ points on it, thereby dividing the circle into $n$ arcs. We then label the obtained arcs as $x_1, \ldots, x_n$ successively in the anticlockwise direction. Based on the type of a relation, we perform the following steps on the circle.
\begin{itemize}
\item If $r_j = x_{j+1} ^{-1} x_j ^{v^{-\epsilon} x_k ^{\epsilon}}$, then we attach the tail of a chord at the point where the arcs $x_k$ and $x_{k+1}$ meet, and attach the head of the same chord at the point where the arcs $x_j$ and $x_{j+1}$ meet. Also, assign sign $\epsilon$ to the chord.
\item If $r_j = x_{j+1}^{-1} x_j ^{v^{\epsilon}}$ and there is no relation of the type $x_{k+1}=x_k^{v^{-\epsilon} x_j^{\epsilon}}$, then put a node with $\epsilon$ sign on the point where arcs $x_j$ and $x_{j+1}$ meet.
\end{itemize}
Further, it is easy to check that the group of the obtained marked Gauss diagram has presentation $\mathcal{P}$.
\end{proof}

As a result of the previous theorem, it is clear that the group $G^*$ in Example \ref{ex:Group-of-general-Gauss-diagram} corresponds to some marked Gauss diagram.

\begin{corollary}
If $G$ has an irreducible $C$-presentation of deficiency $1$, then the group $G^*=G * <v>$ is a $1$-irreducible $C_1$-group of deficiency $2$.
\end{corollary}

\section{Peripheral structures of $\Pi_D$}\label{S: Peripheral subgroup and peripheral structure}
It is well established that the knot group along with a peripheral subgroup and the meridian of a classical knot is a complete invariant in the class of classical knots up to the orientation of the knot and the ambient space. However, this is not the case for virtual knots (see, for example, \cite{Kauffman-1, GPV-1, Kim-1, BB-1}). In particular, if we consider the Kishino knot $K$, then the pair $(G_{\tilde{\varphi}_A}(K); (m,l))$ (refer to Remark \ref{R: one component link}) defined in \cite{BB-1} does not distinguish it from the trivial knot.
\medskip

In this section, we extend the notion of a meridian, longitude, peripheral subgroup and peripheral structure to the groups associated to marked Gauss diagrams and note down their properties modelled on the results of \cite[Subsection 3.3]{Kim-1}
\medskip

Let $D$ be a marked Gauss diagram and fix a base point on the $k^{th}$ circle of $D$ such that it does not lie at the end points of arrows and nodes. Then we define a {\it meridian} $m_k$ to be the generator of $\Pi_D$ corresponding to the arc over which the base point lies. We now describe a procedure to write a {\it longitude} $l_k$ corresponding to the meridian $m_k$. We start moving along the circle from the base point in the anticlockwise direction, and write $v_t^{\epsilon}$ when passing the tail of an arrow whose sign is $\epsilon$ and its head lies on the $t^{th}$ circle, and when passing the head of an arrow we use the following rule:
\begin{itemize}
\item if the sign of the arrow is $+1$, and the end point of its tail is the meeting point of the arcs $x_i$ and $x_{i+1}$ on the $n^{th}$ circle, then we write $v_n^{-1} x_i^{v_k v_n ^{-1}}$.
\item if the sign of an arrow is $-1$, and the end point of its tail is the meeting point of the arcs $x_i$ and $x_{i+1}$ on the $n^{th}$ circle, then we write  $v_n x_i ^{-1}$.
\end{itemize}

Furthermore, if we pass a node, then we write $v_{k}^{\epsilon}$, where $\epsilon$ is the sign of the node. On arriving at the base point, we write $m_k^{-\alpha}$, where $\alpha$ is the sum of signs of arrows whose head lies on the $k^{th}$ circle. Note that $l_k$ is an element of $\Pi_D$.
\par

\begin{remark}
Hereafter, throughout the paper by a marked Gauss diagram we mean a $1$-circle marked Gauss diagram.
\end{remark}
\par

Let $D$ be a given marked Gauss diagram, with $m$ being its meridian and $l$ the corresponding longitude. A \textit{peripheral pair} of a marked Gauss diagram $D$ is the pair $(m, l)$ and the \textit{peripheral subgroup} corresponding to the meridian $m$ of $D$ is the subgroup of $\Pi_D$ generated by $m$ and $l$.
Two pairs $(m,l)$ and $(m',l')$ are said to be \textit{conjugates} if there is an element $g$ in the group $\Pi_D$ such that $m'=m^g$ and $l'=l^g$. We then define the \textit{peripheral structure} as the conjugacy class of a peripheral pair of $D$.
\medskip

We now prove that the peripheral structure of a marked Gauss diagram $D$ is unique and invariant under the marked Reidemeister moves.
Let us consider a presentation $$\Pi_D= \langle x_1, \ldots, x_n, v~~|~~r_1=x_2^{-1}x_1^{w_1}, \ldots, r_n=x_1^{-1}x_n^{w_n} \rangle$$ of $\Pi_D$ which is written as per the procedure described in Section \ref{S: marked Gauss Diagrams}. If $x_1$ is a meridian of $D$, then $l=w_1\ldots w_n x_1^{-\alpha}$ is the corresponding longitude, where $\alpha$ is the sum of signs of chords in $D$.

\begin{proposition}
The peripheral pair and the peripheral subgroup of a marked Gauss diagram are unique up to conjugacy. Moreover, the peripheral structure is invariant under the marked Reidemeister moves.
\end{proposition}
\begin{proof}
Let $D$ be a marked Gauss diagram. We first choose two meridians $m_1$, $m_2$ corresponding to two different arcs of $D$. By construction of the group $\Pi_D$, there exist elements $g_1$ and $g_2$ in $\Pi_D$ such that $m_1=m_2^{g_2}$, $m_2=m_1^{g_1}$, $l_1=g_1g_2m_1^{-\alpha}$ and $l_2=g_2g_1m_2^{-\alpha}$, where $l_1$ and $l_2$ are longitudes corresponding to meridians $m_1$ and $m_2$, respectively. It is not difficult to see that $l_2=l_1^{g_1}$ and so $(m_1, l_1)^{g_1}=(m_2,l_2)$. This implies that the peripheral pair and the peripheral subgroup of $D$ are unique up to conjugacy. Hence peripheral structure of $D$ is independent of choice of meridian. At last, it is easy to check the invariance of the peripheral structure under the marked Reidemeister moves.
\end{proof}

\begin{proposition}
A peripheral subgroup of $\Pi_D$ is abelian.
\end{proposition}
\begin{proof}
Using the previous proposition, it suffices to prove that the subgroup generated by the meridian $x_1$ and the corresponding longitude $l=w_1 \ldots w_n x_1^{-\alpha}$ is abelian. By considering relations in the presentation of $\Pi_D$, we have $x_1=x_1^{w_1 \ldots w_n}$. This implies that the meridian $x_1$ commutes with the longitude $l$.
\end{proof}

Let $G$ be a group presenting a $1$-irreducible $C_1$-presentation $\mathcal{P}=\langle x_1, \ldots, x_n, v~~|~~r_1, \ldots, r_m \rangle$. Let $G_v$ denotes the group with presentation $\langle x_1, \ldots, x_n, v~~|~~r_1, \ldots, r_m, v \rangle $, and for $g \in G$, $g_v$ denotes the image of $g$ in $G_v$. It is easy to observe that for any marked Gauss diagram $D$, the image $l_v$ of longitude $l$ belongs to the commutator subgroup of ${\Pi_D}_v$.

\begin{theorem}\label{sufficiency-of-longitude}
Let $G$ be a group presenting a $1$-irreducible $C_1$-presentation $\langle x_1, \ldots, x_n, v~~|~~r_1, \ldots, r_{n-1} \rangle$ of deficiency $2$ and $l$ an element of $G$. If the image of $l$ in $G_v$ belongs to the commutator subgroup of $G_v$, and $l$ commutes with some conjugate of $x_1$ say $x_0$, then $G$ is the group of a marked Gauss diagram with a peripheral pair $(x_0, l)$.
\end{theorem}

\begin{proof}
Since $x_0$ is conjugate to $x_1$ in $G$, there exists some $w$ in $G$ such that $r_0:=x_1^{-1}x_0^{w}=1$. Thus, $\mathcal{P}=\langle x_0,x_1, \ldots, x_n, v~~|~~r_0,r_1, \ldots, r_{n-1} \rangle$ is a presentation of the group $G$. We may assume that each relation in $\mathcal{P}$ is of the form $r_i=x_{i+1}^{-1}x_i^{w_i}$, $i=0,1, \ldots, n-1$. On adding a redundant relation $r_n=x_0^{-1}x_n^{(w_0 w_1 \ldots w_{n-1})^{-1}l}$ to the presentation $\mathcal{P}$, we get a cyclic irreducible $C_1$-presentation of $G$. By Proposition \ref{irreducible $C_1$-presentation-to-realizable-irreducible $C_1$-presentation}, we can assume that $\mathcal{P}$ is a realizable irreducible $C_1$-presentation. Thus, it is the group of a marked Gauss diagram with a peripheral pair $(x_0, l)$.
\end{proof}

As a consequence, we have the following results. 

\begin{corollary}
Let $G$ be a group with an irreducible $C$-presentation $\langle x_1, \ldots, x_n~~|~~r_1, \ldots, r_{n-1} \rangle$ of deficiency $1$ and $l$ an element of $G$. If $l$ belongs to the commutator subgroup of $G$ and commutes with some conjugate of $x_1$ in $G$ say $x_0$, then $G * \langle v \rangle$ is the group of a marked Gauss diagram with a peripheral pair $(x_0', l')$, where $x_0'$ and $l'$ are natural images of $x_0$ and $l$ in $G * \langle v \rangle$, respectively.
\end{corollary}

\begin{corollary}
Let $G$ be a group with a $1$-irreducible $C_1$-presentation of deficiency $2$. Then $G$ is the group of a marked Gauss diagram with trivial longitude. In particular, if $K$ is a classical knot, then $G_S(K)$ is the group of some marked Gauss diagram with a trivial longitude.
\end{corollary}

\begin{remark}
	In \cite[Corollary $9$]{Kim-1}, it is proved that there exists a non-trivial virtual knot $K$ with a trivial longitude in the group $G_0(K)$. We do not know an example of a non-trivial virtual knot $K'$ having a trivial longitude in the group $G_S(K')$. But using the above corollary, one can construct a non-trivial marked Gauss diagram with a trivial longitude. 
\end{remark}

\medskip 

\section{Peripherally specified homomorphs}\label{S: Peripherally specified homomorphs}
The {\it weight} of a group $G$ is defined as the minimum number of elements required to normally generate $G$. It was asked in \cite{Neuwirth-1} whether a finitely generated group of weight one is a homomorph $($homomorphic image$)$ of a knot group. This was answered positively in \cite{Acuna-1, Johnson-1} where it was proved that for any element $\mu$ in a group $G$ which is finitely generated by the conjugates of $\mu$, there exists a knot $K$ in $\mathbb{S}^3$  and an onto homomorphism $\rho: \pi_1(\mathbb{S}^3 - K) \to G$ such that $\rho(m)=\mu$, where $m$ is a meridian of $K$.
\medskip

Necessary and sufficient conditions for a pair of elements in the symmetric group $S_n$ to be realized as the image of a meridian-longitude pair for some knot $K$ in $\mathbb{S}^3$ can be found in \cite{EL-1}, which was later extended to knot group representations into general groups \cite{JL-1}. Later on,  Kim \cite{Kim-1} extended these results to homomorphs of virtual knot groups $G_0(K)$.
\medskip

In this section, we investigate the following analogous problem.

\begin{problem}
Let $G$ be a group, and $\mu$ and $\nu$ be elements in $G$. Does there exist a marked Gauss diagram $D$ and an onto homomorphism $\rho: \Pi_D \to G$ such that $\rho(m)=\mu$ and $\rho(v)=\nu$, where $m$ is a meridian of $D$ and $v$ corresponds to the component of $D$?
\end{problem}

It is easy to see that $G$ must be finitely generated by $\nu$ and conjugates of $\mu$. We fix $\mu$ and $\nu$ in $G$ with these properties. We say that $\lambda \in G$ is realizable if there is a marked Gauss diagram $D$ and a representation $\rho: \Pi_D \to G$ with above properties such that $\rho(l)=\lambda$, where $l$ is the longitude of $D$ corresponding to the meridian $m$. We denote the set of realizable elements by $\Lambda_G$ and prove the following result.

\begin{theorem}\label{NonEmptyRealizableSet}
The set $\Lambda_G$ is non-empty.
\end{theorem}
\begin{proof}
Let $G=\langle \mu_1, \ldots, \mu_n, \nu \rangle$, where $\mu_1=\mu$ and $\mu_i=\mu_1^{w_i}$, $2\leq i \leq n$ and $w_i$ are words in $\mu_1^{\pm 1}, \ldots, \mu_n^{\pm 1}, \nu^{\pm 1}$. Using the techniques described in the proofs of theorems \ref{irreducible $C_1$-presentation-to-cyclic irreducible $C_1$-presentation} and \ref{cyclic irreducible $C_1$-presentation-to-realizable-irreducible $C_1$-presentation}, we can assume the following in $G$.
\begin{itemize}
\item $\mu_{i+1}=\mu_i^{w_i}$ for $ i \in \mathbb{Z}_n$.
\item Each $w_i$ is either $\nu^{\epsilon}$ or $\nu^{-\epsilon} \mu_j^{\epsilon}$, where $\epsilon=\pm1$. If $w_k= \nu^{-1} \mu_p$, then $w_p=\nu$ and if $w_k= \nu \mu_p^{-1}$, then $w_p=\nu^{-1}$. Moreover, $w_j \neq \nu^{-\epsilon} \mu_p^{\epsilon}, \nu^{-\epsilon} \mu_k^{\epsilon}$, where $k \neq j$.
\end{itemize}
By Theorem \ref{RealizationOfirreducible $C_1$-presentation}, we construct a marked Gauss diagram $D$ corresponding to the following realizable irreducible $C_1$-presentation
$$
\langle x_1, \ldots, x_n, v~~|~~x_{2}^{-1}x_1^{u_1}, \ldots, x_1^{-1}x_n^{u_n} \rangle,
$$ where $u_j$ is obtained by replacing $\mu_i$ and $\nu$ in $w_j$ with $x_i$ and $v$, respectively for every $1 \leq i,j \leq n$. Clearly, we have a well-defined onto homomorphism $\rho: \Pi_D \to G$ mapping $x_1$ to $\mu_1$ and $v$ to $\nu$.
\end{proof}

\par
Let us now consider two marked Gauss diagrams $D_1$ and $D_2$. Let $p_1$ and $p_2$ be two points on $D_1$ and $D_2$, respectively, which are not meeting any chord or node. The \textit{connected sum} $D$ of $D_1$ and $D_2$ at $p_1$ and $p_2$ is a marked Gauss diagram obtained by removing a small interval around $p_1$ and $p_2$ not intersecting a chord or a node, and then joining the end points of remaining diagrams while respecting the orientations. Let $\Pi_{D_1}=\langle x_1, \ldots, x_n, v_1~~|~~x_2^{-1}x_1^{\alpha_1}, \ldots, x_1^{-1}x_n^{\alpha_n} \rangle$ and
$\Pi_{D_2}=\langle y_1, \ldots, y_m, v_2~~|~~y_2^{-1}y_1^{\beta_1}, \ldots, y_1^{-1} y_m^{\beta_m} \rangle$ be group presentations of $D_1$ and $D_2$, respectively. If $p_1$ and $p_2$ are on the arcs $x_1$ and $y_1$, respectively, then a group presentation of $\Pi_D$ is
$$\langle x_1, \ldots, x_n, y_1, \ldots, y_m, v~~|~~ x_2^{-1} x_1^{{\alpha}'_1}, \ldots, x_n^{-1}x_{n-1}^{\alpha'_{n-1}}, y_1^{-1} x_n^{\alpha'_n}, y_2^{-1}y_1^{\beta'_1}, \ldots, y_m^{-1}y_{m-1}^{\beta'_{m-1}}, x_1^{-1}y_m^{\beta'_m} \rangle.$$ Note that here $\alpha'_i$ and $\beta_j'$ are obtained from $\alpha_i$ $( 1 \leq i \leq n)$ and $\beta_j$ $(1 \leq j \leq m)$ by replacing $v_1$ and $v_2$ by $v$.

\begin{example}
Figure \ref{F: ConnectedSum} shows the marked Gauss diagram $D_3$ as a connected sum of two non-trivial marked Gauss diagrams $D_1$ and $D_2$ at $p_1$ and $p_2$. It is easy to check that $\Pi_{D_3} \cong \mathbb{Z} * \mathbb{Z}$. Let us now consider $T$ to be a marked Gauss diagram without chords and nodes, then $\Pi_T \cong \mathbb{Z} * \mathbb{Z}$. However, it can be easily seen that $D_3$ is not equivalent to $T$.

\begin{figure}[H]
	\centering
	\includegraphics[scale=1.2]{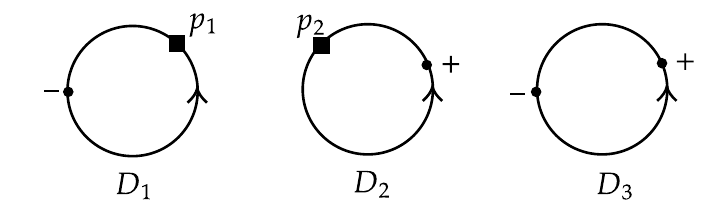}
\caption{$D_3$ as the connected sum of $D_1$ and $D_2$ at points $p_1$ and $p_2$.} \label{F: ConnectedSum}
\end{figure}
\end{example}

Our final result shows that the non-empty set $\Lambda_G$ is, in fact, a subgroup of $G$.

\begin{theorem}\label{lambda-is-a-subgroup}
The set $\Lambda_G$ is a subgroup of $G$.
\end{theorem}
\begin{proof}
Let $\lambda_1, \lambda_2$ be two elements of $\Lambda_G$. Since $\lambda_1, \lambda_2$ are realizable, there exist marked Gauss diagrams $D_1$, $D_2$ and onto homomorphisms $\rho_1: \Pi_{D_1} \to G$, $\rho_2: \Pi_{D_2} \to G$ such that $\rho_1(x_1)=\rho_2(y_1)=\mu$, $\rho_1(v_1)=\rho_2(v_2)=\nu$ and $\rho_1(l_1)=\lambda_1$, $\rho_2(l_2)=\lambda_2$, where $(x_1,l_1)$ and $(y_1,l_2)$ are meridian-longitude pairs of $D_1$ and $D_2$, respectively. Let $D$ be a connected sum of $D_1$ and $D_2$ made over the points lying on the arcs $x_1$ and $y_1$. Define a map $\rho: \Pi_D \to G$ such that $\rho(x_i)=\rho_1(x_i)$, $\rho(y_j)=\rho_2(y_j)$ and $\rho(v)=\nu$ for every $1 \leq i \leq n$ and $1 \leq j \leq m$. It is easy to see that the map $\rho$ is a well-defined onto homomorphism and that $\rho(l)=\lambda_1 \lambda_2$, where $l$ is the longitude corresponding to the meridian $x_1$ for $D$. Let $\overline{D}_1$ be the marked Gauss diagram obtained from $D_1$ by reversing the orientation of the circle and signs of chords and nodes. If $\bar{x}_i$ denote the arc in $\overline{D}_1$ which was labelled $x_i$ in $D_1$, then the map $\bar{\rho}_1: \Pi_{\overline{D}_1} \to G$ defined by $\bar{\rho}_1(\bar{x}_i)=\rho_1(x_i)$ is well-defined and $\bar{\rho}_1(l)=\lambda_1^{-1}$, where $l$ is the longitude of $\overline{D}_1$ corresponding to the meridian $\bar{x}_1$ for $\overline{D}_1$. This completes the proof.
\end{proof}

\medskip

\section{Summary and questions}\label{S: Problems}
In this paper, we have mainly introduced the notion of marked Gauss diagrams and marked virtual link diagrams. The motivation behind marked Gauss diagrams is to extend the domain of virtual link groups introduced in this paper to knot-like diagrams.
We have studied the algebraic results of the groups associated to marked Gauss (link) diagrams by keeping the paper \cite{Kim-1} as an important point of reference. Here we collect some open questions for future work in this area.

\begin{itemize}
	\item Does there exist a representation of $VB_n$ which is not equivalent to a virtually symmetric representation? If the answer is no, then in some sense, it means that virtual crossings do not affect the invariants arising from virtual braid group representations. 
	\item It would be interesting to compare the peripheral structures of virtual link groups introduced in this paper with the peripheral structure of virtual link groups introduced by Kauffman \cite{Kauffman-1} in terms of their ability to distinguish virtual links.
	\item Analogous to the concept of abstract link diagrams for virtual links (see \cite{KK-1}), we expect that marked virtual link diagrams can be interpreted as marked link diagrams on surfaces. It would be interesting to know whether marked virtual link diagrams can be interpreted as some kind of codimension two embeddings.
	\item Under what conditions an $m$-irreducible $C_m$-group, $m \geq 2$, can be realized as the group of a marked Gauss diagram?
\end{itemize}

\begin{ack}
Valeriy G. Bardakov is supported by Ministry of Science and Higher Education of Russia (agreement No. 075-02-2021-1392) and the Russian Science Foundation grant 19-41-02005. Mikhail V. Neshchadim is supported by the Russian Science Foundation grant 19-41-02005. Manpreet Singh was supported by IISER Mohali for the PhD Research fellowship. Manpreet Singh also thanks to his supervisor Dr. Mahender Singh for giving him the opportunity to attend VI Russian-Chinese Conference on Knot Theory and Related Topics at NSU (Novosibirsk) and 2nd International Conference on Groups and Quandles in low-dimensional topology at TSU (Tomsk) using his grant, where he had discussions with the first two authors. His visit to Russia was supported by the DST grant INT/RUS/RSF/P-02.
\end{ack}

\end{document}